\documentclass[reqno, 11pt, a4paper]{amsart}
\usepackage{amsmath, amssymb, amsthm}  
\usepackage{cancel}
\usepackage{enumitem}
\usepackage{enumerate}
\usepackage{mathrsfs}
\usepackage{mathtools}
\usepackage{hyperref}
\usepackage[capitalise]{cleveref}

\usepackage[percent]{overpic}

\usepackage{graphicx, color}
\numberwithin{equation}{section}
\usepackage{a4wide}
\usepackage[active]{srcltx}
\usepackage[table]{xcolor}
\definecolor{cyan(process)}{rgb}{0.0, 0.72, 0.92}
\usepackage{xcolor}
\definecolor{columbiablue}{rgb}{0.61, 0.87, 1.0}
\usepackage{wrapfig}
\usepackage[]{pstricks}
\usepackage{fancyhdr}
\definecolor{sandstone}{HTML}{786D5F}
\definecolor{beaublue}{rgb}{0.74, 0.83, 0.9}
\definecolor{cherryblossompink}{rgb}{1.0, 0.72, 0.77}
\definecolor{light-gray}{gray}{0.95}

\usepackage{tikz}
\usetikzlibrary{shapes,arrows,positioning,fit,backgrounds,calc}
\usepackage{pgfplots}
\pgfplotsset{compat=1.16}

\newcommand{\vertiii}[1]{{\left\vert\kern-0.25ex\left\vert\kern-0.25ex\left\vert #1 \right\vert\kern-0.25ex\right\vert\kern-0.25ex\right\vert}}

\newcommand{\E}{\mathbb{E}}

\usepackage[normalem]{ulem}

\usepackage[truedimen,top=3truecm, bottom=3truecm, left=2.5truecm, right=2.5truecm, includefoot]{geometry}

\begin{document}

\title[Stochastic Oscillators out of equilibrium]{Stochastic Oscillators out of equilibrium:\\ scaling limits and correlation estimates}

\author[P. Gon\c{c}alves]{Patr\'{i}cia Gon\c{c}alves}
\address{Center for Mathematical Analysis, Geometry and Dynamical Systems, Instituto Superior T\'ecnico, Universidade de Lisboa}
\email{pgoncalves@tecnico.ulisboa.pt}

\author[K. Hayashi]{Kohei Hayashi}
\address{Department of Mathematics, Graduate School of Science, The University of Osaka \\
\emph{and} 
RIKEN Center for Interdisciplinary Theoretical and Mathematical Sciences}
\email{khayashi@math.sci.osaka-u.ac.jp}

\author[J. Mangi]{João Pedro Mangi}
\address{Instituto de matem\'atica pura e aplicada \\
\emph{and} Center for Mathematical Analysis, Geometry and Dynamical Systems, Instituto Superior T\'ecnico, Universidade de Lisboa}
\email{mangi.joao@impa.br}

\begin{abstract} 
We consider a purely harmonic chain of oscillators which is perturbed by a stochastic noise. Under this perturbation, the system exhibits two conserved quantities: the volume and the energy. At the level of the hydrodynamic limit, under diffusive  scaling, we show that depending on the strength of the Hamiltonian dynamics, energy and volume evolve according to either a system of autonomous heat equations or a non-linear system of coupled parabolic equations. Moreover,  for general initial measures, under diffusive scaling, we can characterize the non-equilibrium  volume fluctuations.  The proofs are based on precise bounds on the two-point volume correlation function and a uniform fourth-moment estimate. 
\end{abstract}

\keywords{Bernardin-Stoltz model, Harmonic oscillator, Hydrodynamic Limits, Non-equilibrium fluctuation, Ornstein-Uhlenbeck process, Edwards-Wilkinson universality}
\subjclass[2020]{60H15; 60J27; 82C22; 60K50}
\maketitle

\theoremstyle{plain}
\newtheorem{theorem}{Theorem}[section] 
\newtheorem{lemma}[theorem]{Lemma}
\newtheorem{corollary}[theorem]{Corollary}
\newtheorem{proposition}[theorem]{Proposition}

\theoremstyle{definition}
\newtheorem{definition}[theorem]{Definition}
\newtheorem{remark}[theorem]{Remark}
\newtheorem{assumption}[theorem]{Assumption}
\newtheorem{example}[theorem]{Example}


\makeatletter
\renewcommand{\theequation}{%
\thesection.\arabic{equation}}
\@addtoreset{equation}{section}
\makeatother

\makeatletter
\renewcommand{\p@enumi}{A}
\makeatother


\newcommand{\acks}{\textbf{Acknowledgements.}}

\newenvironment{acknowledgements}{%
\renewcommand{\abstractname}{Acknowledgements}
\begin{abstract}
}{%
\end{abstract}
}

\allowdisplaybreaks

\section{Introduction}
The derivation of the macroscopic evolution of conserved quantities in stochastic chains of oscillators has received a lot of attention in the last decade; see, e.g. \cite{bernardin2007hydrodynamics, bernardin2014anomalous, bernardin2012anomalous, gonccalves2014nonlinear, gonccalves2020derivation,jara2015superdiffusion, gonccalves2023derivation,komorowski2020kinetic}, just to mention a few. These are toy models for a variety of phenomena, and the rigorous understanding of their behavior is still very limited. 
Over the last years, by investigating the large-scale behavior, namely hydrodynamic limits and fluctuations of their conserved quantities, many important mathematical tools have been derived, which allow also for a mathematically rigorous treatment of super-diffusive systems.

\subsection*{The Bernardin-Stoltz model}

In this work, we consider the so-called Bernardin-Stoltz (BS) model introduced in \cite{bernardin2012anomalous}, which is a model that presents several analogies with the standard chains of oscillators, but has a simpler mathematical treatment. The dynamics of the chain of oscillators is a purely deterministic Hamiltonian dynamics given in terms of some potentials.  More precisely, let $U,V:\mathbb R \to \mathbb R_+$ be two potentials and consider the Hamiltonian system $\{(r_t,p_t):t\geq0\}$ whose time evolution is given by 
\begin{equation*}
\begin{cases}
\begin{aligned}
& \frac{dp_t(x)}{dt}=V'(r_t(x+1))-V'(r_t(x)), \\
& \frac{dr_t(x)}{dt}=U'(p_t(x+1))-U'(p_t(x)).
\end{aligned}
\end{cases}
\end{equation*}
Above, $p_t(x)$ is the momentum of particle $x$,  $r_t(x)=q_t(x)-q_t(x-1)$ is the deformation where $q_t(x)$ is its position, at time $t$. 
Moreover, $x$ is interpret as a discrete space variable - throughout the present paper, we consider a one-dimensional discrete torus $\mathbb T_N$ where $N\in\mathbb N$.   
The BS model $\{\eta_t:t\geq0\}$ is obtained by setting $U=V$, and, by mapping $\eta_t(x)$ to variables $p_t(x)$ and $r_t(x)$ in the following way: $\eta_t(2x)=p_t(x)$ and $\eta_t(2x-1)=r_t(x)$. From this, we get an $N$-system of oscillators whose (deterministic) time evolution is governed by the following Hamiltonian dynamics:
\begin{equation*}
\frac{d}{dt} \eta_t(x)
= V'\big(\eta_t(x+1)\big) 
- V'\big(\eta_t(x-1)\big)  .
\end{equation*}

In addition to the above deterministic evolution,  \cite{bernardin2012anomalous} introduced a noise effect where a randomly chosen pair of oscillators exchange their states after exponential random clocks that are dictated by Poisson clocks attached to nearest-neighbor chains. We assume that the rates of these Poisson clocks are uniform in space, so that the above interaction defines a nearest-neighbor, spatially homogeneous stochastic dynamics. 
Under this time evolution, the following quantities
\begin{equation*}
\sum_{x} \eta_t(x) \quad\text{and}\quad 
\sum_x V(\eta_t(x)),
\end{equation*}
which are referred to as \emph{total volume} and \emph{total energy}, respectively, are conserved.

A fundamental question is to derive the hydrodynamic limit for the two conserved quantities, i.e. the system of partial differential equations governing the space-time evolution of the volume and energy of the system and to understand the behavior of the fluctuations of the system around this limit. To that end, we scale the spatial variable by $1/N$ and accelerate the Hamiltonian dynamics (resp. exchange noise) by $N^{a-\kappa}$ (resp. $N^a$) with some $a\ge 0$ and $\kappa\ge 0$, and our goal is to describe the limiting macroscopic behavior as $N\to \infty$ associated with each conserved quantity, for each choice of the potential and the parameters $a$ and $\kappa$. 
Now we review some previous results regarding the limiting behavior of the BS model both at the level of the hydrodynamic limit, as well as the fluctuations of the microscopic system around this limit.

\subsection*{Scaling limits}
In \cite{bernardin2007hydrodynamics}, the hydrodynamic limit was obtained for the system of harmonic oscillators, but the perturbing noise is continuous, in which case the system presents a particular mathematical property called the \emph{fluctuation-dissipation} relation. However, we observe that the aforementioned property is not available in the BS model with the exchange noise.  At the level of fluctuations, many remarkable results were obtained at stationarity, i.e. when  the system starts from the Gibbs invariant measure. A significant result regarding the equilibrium fluctuations was obtained when the potential is purely harmonic, i.e., $V(\eta) =\eta^2/2$. Since there are two conservation laws, we note that the volume fluctuation behaves, under the diffusive time-scaling (i.e. $a=2$) as an Ornstein-Uhlenbeck process; though  the energy fluctuations cross from the  so-called $3/2$-L\'evy anomalous diffusion, observed in the time scale corresponding to $a=3/2$ and cross to the Ornstein-Uhlenbeck process when the intensity of the Hamiltonian is weaker when compared to the exchange dynamics. For the interested readers, we refer to \cite{bernardin20163, bernardin2018weakly, bernardin2018nonlinear, gonccalves2023derivation}, see also \cite{bernardin2018interpolation}. 

Another significant choice of the potential is the so-called Kac-van Moerbeke potential (or Toda lattice potential in some literature), i.e., $V(\eta)= e^{-b \eta}-1 + b\eta $ for some $b>0$. 
We note that in \cite{bernardin2014anomalous}  an anomalous behavior for the  energy fluctuations was observed. Later in \cite{ahmed2022microscopic} it was proved that the fluctuations for another quantity -  obtained from an appropriate linear combination of volume and energy - are given by the energy solution of the stochastic Burgers equation.
This suggests that if the nonlinearity of the potential is strong enough, there is a proper coupling of volume and energy, which belongs to the so-called Kardar-Parisi-Zhang universality class. 
More recently in \cite{gonccalves2023derivation}  the BS model was studied under a  generic potential, and by properly coupling the two conserved quantities, it was derived both the energy solution  of the stochastic Burgers equation, as well as, the $3/2$-L\'evy anomalous diffusion. 

Besides the equilibrium fluctuation of the BS model, although the results in the last decade revealed that it already includes intriguing phenomena, extension to the non-equilibrium fluctuation is undoubtedly a significant topic. 
Nevertheless, there are very few results on the non-equilibrium fluctuations of the BS model, and the main reason could rely on the fact that the model presents some pathologies for the present techniques available in the literature to be employed. These can be numbered as follows: first,  it is a Hamiltonian system with an unbounded state space; the variables can take values in $\mathbb R$; the two conserved quantities do not evolve autonomously and each can evolve in its own time scale; the Hamiltonian dynamics increases the degree of the input function while the exchange noise preserves it. 

\subsection*{Our results}
In this paper, we consider the BS model with a purely harmonic potential, i.e., $V(\eta)=\eta^2/2$, and  an exchange noise. We restrict ourselves to the one dimensional torus with $N$ sites. Moreover, we tune the intensity of the Hamiltonian dynamics by a factor $\alpha_N=\alpha N^{-\kappa}$ and we take the system in the diffusive time scale $N^2$. Above  $\alpha\in\mathbb R$ and $\kappa\geq 0$. We prove, for any $\kappa>1$, the hydrodynamic limit for volume and energy, and the evolution equations are given by two autonomous heat equations, namely:
\begin{equation*}
\begin{cases}
\partial_t v=\Delta v\\
\partial_t e=\Delta e.
\end{cases}
\end{equation*}
We also prove that in the critical case $\kappa=1$, the volume and energy evolve according to the system of equations:
\begin{equation*}
\begin{cases}
\partial_t v=\Delta v+2\alpha\nabla v\\
\partial_t e=\Delta e+2\alpha\nabla v^2.
\end{cases}
\end{equation*}

As we mentioned above, these models present several pathologies to fit the existing methods in the literature. The volume hydrodynamics holds in a stronger sense when compared to the energy hydrodynamics, since in the latter case we can prove results in expectation, while in the former, they hold in probability. In both cases, we do not have for granted bounds on the second moments of both volume and energy. Nevertheless,  we are able to derive a  fourth moment bound, obtained through an analysis of a discrete $\mathcal H^{-1}$-norm of the process.  The hydrodynamic limits for the volume are established under the entropy method. Since the variables take values in $\mathbb R$ we are then forced to consider the volume empirical measure  as an element in some Sobolev space. To make the presentation coherent we also consider the energy empirical measure  acting on the same Sobolev space. 

The reason why the hydrodynamic limit for the energy is presented only in expectation (in the case $\kappa=1$) is because the evolution equation involves a factor given in terms of the volume that has a non trivial limit in the same  time scale as the energy. This brings difficulties at the microscopic level in order to close the equation for the energy in terms of the square of the volume. 
Therefore, we prove that the convergence in expectation but we are also able to prove it in a much stronger sense, i.e. we prove that the discrete energy profile converges to the solution of the hydrodynamic equation in the $L^\infty$ sense and we quantify the rate of convergence in terms of the parameter $\kappa$ ruling the strength of the Hamiltonian dynamics. The main reason for this partial result  is the fact that the variables are real valued, and the usual techniques for proving the one- and two-block replacement lemmas do not apply in this case. To circumvent this, we embrace on the  analysis of higher-order correlation functions to replace the microscopic variables by their discrete macroscopic counterparts. This will be explained in more detail ahead. 

We also prove that the volume fluctuation field converges to a limit that we characterize through \eqref{eq:volume_limit_field} when the system starts from generic measures.  The main step of our proof is a precise estimate of the two-point correlation function for the volume. This estimate is obtained by first deducing an evolution equation for the correlation function written in terms of the generator of a random walk evolving on the two-dimensional discrete torus apart  diagonal points $x=y$. From Duhamel's principle we can then write the correlation function in terms of this random walk.  We note that in order to control the contribution of the  exchange dynamics to that evolution equation, we need to control the local time that the two-dimensional random stands at the line $y=x+1$ and this we can show to be of order $1/N$. At this point, the Hamiltonian dynamics brings more difficulties since it introduces extra terms in the equation that  require a much finer analysis. These are written in terms of energy and volume discrete profiles, but, while we do have good bounds for the volume discrete profile - as it solves the discrete heat equation and maximum principles are available -  we do not have apriori estimates for the energy discrete profile. To circumvent this difficulty, we derive an evolution equation for these extra terms in order to obtain good bounds. And again from Duhamel's principle we write and evolution equation in terms of a one-dimensional random walk. However again here the Hamiltonian dynamics brings the correlation function into the evolution equation for the energy profile. Our strategy is then to iterate the procedure and optimize our bounds on short times and long times by using sharp bound on the transition probabilities for the underlying one dimensional random walk.

We note that  the analysis of the evolution  equation of the correlation function to prove non-equilibrium fluctuations was deeply explored in the context of the exclusion process where the stirring process is used instead of the random walk, see for instance \cite{landim2008stationary,gonccalves2017nonequilibrium,franco2019non,gonccalves2020non,franceschini2023non}. 
However, as mentioned above, due to the presence of the Hamiltonian dynamics, the equation behaves quite differently from the ones of the exclusion process. 
We also note that the full characterization of the limit of the volume fluctuation field is not obtained here for the same reason as we do not obtain the hydrodynamics for the energy as we mentioned above.  The reason is that we are not able to characterize the law of the limit martingale apart when the system is at stationarity in which case we can show a strong convergence result for the quadratic variation which then characterizes the noise appearing in the equation of the volume field. In order to carry out this program to the general scenario we could rely on deriving four point correlation estimates. If it were not the Hamiltonian dynamics, this would be a simple problem, but the presence of the Hamiltonian dynamics brings extra difficulties: in one hand we cannot use the stirring process because our variables are unbounded and moreover, the random walks that appear when considering the corrrelation function for more than two points have non-homogeneous rates and move to non-nearest-neigbors and sharp bounds on their transition probabilities seem hard to obtain. We leave this as an open problem.

\subsection*{Related works and open problems} 
We note that a similar model of stochastic chains of oscillators has also been considered previously in the literature. In the sequel of papers \cite{komorowski2020fractional,komorowski2020kinetic,komorowski2021thermal, KOS2018macroscopic, komorowski2021hydrodynamic}, the authors studied the evolution of mechanical and thermal energy and showed a similar macroscopic behavior in the scaling limit. In \cite{jara2015superdiffusion}, it was also proved some super-diffusive behavior for the energy at equilibrium. At the technical level, their proof is based of the analysis of the Fourier-Wigner distributions, being completely different from the techniques used in the present paper. Moreover, the results proved in this article concerning the hydrodynamic limit are a generalization of the ones obtained in \cite{komorowski2021hydrodynamic,KOS2018macroscopic}.
We now discuss some other open problems concerning the harmonic BS model. One major concern not only in the context of non-equilibrium fluctuations but also at the level of hydrodynamics and other questions in statistical mechanics is the quantification of the decay of higher-order correlation functions. In this article, we give a first step in this direction as we prove that the 2-point correlation function for the volume decays as $1/N$. The analogue result for the energy is left open. Also higher order correlation estimates are left open. As we mentioned above, the main difficulty is to deal with the contribution of the Hamiltonian dynamics in the evolution equations for these functions and to analyze the corresponding transition probabilities of the underlying random walks. We note that in the case of the exclusion process, the situation is simpler as one can write the evolution equations in terms of the stirring process that then can be compared to independent particles and from the local CLT one can then use heat kernel estimates. Here, the situation is more intricate because of the presence of the Hamiltonian dynamics.

Now, in what concerns energy correlations, the picture is much poorly understood. We note that our strategy would give some hierarchical structure for the evolution equation of the energy correlation functions, albeit other quantities (not conserved) would appear for which we do not have any control. We also leave this as an open problem. We note that when carrying out this program for estimating the $k$ point correlation function it would be an interesting problem to deduce rates of convergence in the same spirit as recently done for the exclusion in \cite{gess2024quantitative}.

\subsection*{Organization of the paper} 
The paper is organized as follows. In Section \ref{sec:model main results} we define the model and present the statements of the main results of the article, namely Theorems \ref{thm:HDL weak regime}, \ref{thm:HDL_volume} and \cref{thm:HDL_energy} for the hydrodynamic limit, and \ref{thm:main fluctuations} for the non-equilibrium fluctuation.
Then, the forthcoming \cref{sec:hydrodynamics} (resp. \cref{sec:fluctuations}) will be devoted to show the hydrodynamics limit (resp. fluctuations).
To show the main theorems, we have two crucial estimates: fourth-moment bound and correlation estimates, which are listed in \cref{thm:fourth-moment_bound} and \cref{lem:correlation_estimate}, respectively. 
One of them is a fourth-moment bound of the process, which allows for instance to control the quadratic variation of the energy of the process and this is an important step in the proof of Theorems \ref{thm:HDL weak regime} and \ref{thm:HDL_volume}. 
The proof of the fourth-moment bound is postponed to \cref{sec:moment bound}.
On the other hand, we show the volume-volume correlation estimate in Section \ref{sec:two-point_correlation_estimate}.  
This is the main ingredient in the proof of Theorems \ref{thm:HDL_volume} and \ref{thm:main fluctuations}, and is also a significant contribution of the paper. 
In the end, for completeness for the readers, we give some explicit computations involving the generator of the process and some random walk estimates in Appendix.

\subsection*{Notation}
Let $\mathbb T_N = \mathbb Z/ (N\mathbb Z)\cong\{0,\ldots, N-1\}$ be the one-dimensional discrete torus and let $\mathbb T=\mathbb R/\mathbb Z\cong [0,1)$ be the usual one-dimensional unit torus.   
For each real sequence $g = (g_x)_{x\in\mathbb T_N}$, define the following discrete operators:
\begin{equation}
\label{eq:discrete_operators_definition}
\begin{aligned}
\nabla^N g_x = N(g_{x+1}-g_x),\quad
\widetilde{\nabla}^N g_x = \frac{N}{2}(g_{x+1}-g_{x-1}) .
\end{aligned}
\end{equation}
Moreover, for any integer $d\ge 1$, the $d$-dimensional Laplacian is defined as follows: 
\begin{equation*}
\Delta^{(d)}_N g(x) 
= N^2 \sum_{i=1}^{2d} \big( g(x+e_i)-g(x) \big) 
\end{equation*}
for each $x\in\mathbb T^d_N$, where $(e_i)_{i=1,\ldots,d}$ stands for the standard unit vectors and we set $e_{i+d}=-e_i$. 
Let us use the short-hand notation for the one-dimensional case: $\Delta^N=\Delta_N^{(1)}$. 
We denote by $L^2(\mathbb T)$ the space of square integrable measurable functions, and denote its norm and scalar product, respectively, by $\|\cdot\|_2$ and $\langle\cdot,\cdot\rangle$.
Finally, we note here that we use the letter $C$ to denote a constant appearing in estimates as multiplication, and we use the same letter for all inequalities by abuse of notation, so that the constant $C$ may be different from line to line.   

\subsection*{List of random walks}
As a main ingredient of this paper, we will present some estimates on several random walks. 
Here, in order to clarify the notation, let us list up three random walks that we will analyze in the main part. 
\begin{itemize}
\item Let $\{X^{(2)}_t:t\ge 0\}$ be a two-dimensional random walk, which is reflected on the diagonal line and generated by an operator $\mathscr L^{(2)}_n$ (see \eqref{eq:2d_RW_operator_definition} for its definition) on the two-dimensional torus without the diagonal line, that is, a space $V^{(2)}_N$ (see \eqref{eq:2d_torus_minus_diagonal_definition}) starting from any $(x,y)\in V^{(2)}_N$. The corresponding probability measure is denoted by $\mathbf P_{(x,y)}$ and we write $\mathbf E_{(x,y)}$ for the expectation with respect to $\mathbf P_{(x,y)}$.
\item We will map the two-dimensional random walk $\{X^{(2)}_t: t\ge 0\}$ onto a one-dimensional random walk evolving in the space $\Lambda_{N-1}=\{0,\ldots, N-1\}$. This projected random walk will be denoted by $\{\widetilde X^{(1)}_t: t\ge 0\}$ and its generator by $\mathfrak L^{(1)}_N$ (see \eqref{eq:1d_RW_after_projection_generator} for the definition), starting from any $x\in \Lambda_{N-1}$. 
\item Finally, let $\{ Y^{(1)}_t:t\ge0 \}$ be a one-dimensional random walk on the torus $\mathbb T_N$, generated by the one-dimensional Laplacian $\Delta^{(1)}_N$, starting from any $x\in\mathbb T_N$. The associated probability measure (resp. its expectation) is denoted by $P_x$ (resp. $E_x$).
\end{itemize}
Note that all random walks we listed above are accelerated by $N^2$, although we omit the dependency on it by a bit of abuse of notation, just for simplification.

\section{Model description and main results}
\label{sec:model main results}
\subsection{The microscopic model}
\label{subsec:model}
Here we detail properties of the dynamics that we will analyze. 

\subsubsection{The dynamics}
In what follows, we consider the dynamics on the torus $\mathbb T_N$ as in \cite{ahmed2022microscopic}.
Let $\Omega_N=\mathbb{R}^{\mathbb{T}_N}$ denote the space of oscillators and we will denote elements of $\Omega_N$ by $\eta=\{\eta(x)\}_{x\in\mathbb T_N}$. 
In what follows, we will consider a Markov process $\{\eta_t(x)\}_{t\ge0}$ defined on $\Omega_N$. 
For that purpose, now let us describe the generator of the process. 
For a smooth bounded function $f:\Omega_N\to\mathbb{R}$, we define  
\begin{equation*}
Sf(\eta)=\sum_{x\in\mathbb{T}_N} \big( f(\eta^{x,x+1})-f(\eta)\big) 
\end{equation*}
and 
\begin{equation*}
Af(\eta)=\sum_{x\in\mathbb{T}_N}\big( \eta(x+1)-\eta(x-1) \big) \partial_xf(\eta)
\end{equation*}
where $\partial_x=\partial/\partial \eta(x)$ denotes the partial derivative with respect to $\eta(x)$, and $\eta^{x,y}$ denotes the configuration in $\Omega_N$ given by
\begin{equation*}
\eta^{x,y}(z)=
\begin{dcases}
\eta(x) \quad\text{if }z=y, \\
\eta(y) \quad\text{if }z=x, \\
\eta(z) \quad\text{if }z\neq x,y
\end{dcases}
\end{equation*}
for each $x,y\in\mathbb T_N$ and each $\eta\in\Omega_N$.  
Then, we define the operator $L_N$ by
\begin{equation*}
L_N=N^2S + \alpha_N N^2A
\end{equation*}
where
\begin{equation}
\label{eq:weak_asymmetry}
\alpha_N = \alpha N^{-\kappa}
\end{equation}
for some $\alpha >0$ and $\kappa>0$. 
Note that the parameter $\alpha_N$ is vanishing as $N\to\infty$ and it introduces the \textit{weak asymmetry} in the dynamics. Moreover, note that the process has two conservation laws, namely, 
\begin{equation*}
\sum_{x\in\mathbb T_N}\eta(x) 
\quad \text{ and }\quad 
\sum_{x\in\mathbb T_N}\eta(x)^2,
\end{equation*}
which are referred to as the total volume and energy, respectively. \par
For a metric space $\mathbb{X}$, let $D([0,T],\mathbb{X})$ denote the space of right continuous and with left limits paths from $[0,T]$ to $\mathbb{X}$ endowed with the Skorohod topology and the associated Borel $\sigma-$algebra. From now on we will denote by $\{\mu_N\}_{N\in\mathbb N}$ any sequence of probability measures on the state space $\Omega_N$ and $\{ \eta_t(x):x\in\mathbb T_N, t\ge 0 \}$ denotes the Markov process on $\Omega_N$ with generator $L_N$, starting from $\mu_N$, where we omit the dependence on $N$ by abuse of notation for the process.  
Let $\mathbb P_N$ denote the probability measure on $D([0,T],\Omega_N)$ induced by the continuous-time Markov process associated to the generator $L_N$ and the initial measure $\mu_N$. Denote by $\mathbb E_N$ the expectation with respect to the measure $\mathbb P_N$.

\subsubsection{Invariant measures}
The process $\{\eta_t(x):t\ge 0, x\in\mathbb T_N\}$ has a family of invariant measures $\{\mu_{\rho,\beta}:\rho\in\mathbb R,\beta>0\}$ where $\mu_{\rho,\beta}$ is a product of Gaussian measures whose common marginal satisfies  
\begin{equation*}
\mu_{\rho,\beta}(\eta(x) \le y)
=
\int_{-\infty}^y \sqrt{\frac{\beta}{2\pi}}\exp{\Big(-\frac{\beta}{2}(z-\rho)^2\Big)}  dz  
\end{equation*}
for each $x\in\mathbb T_N$ and $y\in\mathbb R$.  
Clearly, under the stationary measure we have for any $x\in\mathbb T_N$ that 
\begin{equation*}
E_{\mu_{\rho,\beta}}[\eta(x)] = \rho,\quad
E_{\mu_{\rho,\beta}}[\eta(x)^2] = 1/\beta + \rho^2. 
\end{equation*}

\subsection{Main results}
Now that we have introduced the dynamics, we present our main results. We start by describing the possible initial measures of the process.

\subsubsection{Assumptions on the initial measure}
In order to properly describe the generality of our results, we introduce now the discrete volume and energy profiles, which are defined for each $t \ge 0$ and $x\in\mathbb T_N$, by 
\begin{equation}
\label{eq:element_definition}
v^N_t(x) 
= \mathbb{E}_N[\eta_t(x)],\quad
e^N_t(x) 
= \mathbb E_N[\eta_t(x)^2]. 
\end{equation}
We also introduce another function that will be of fundamental importance in our proofs.
For $t \ge 0$ and $x,y\in\mathbb Z$, let $\varphi_t^N(x,y)$ denote the two-point correlation function for the volume which is given by
\begin{equation}
\label{eq:correlation_volume_definition}
\varphi^N_t(x,y)
=\E_N[\overline\eta_t(x)\overline\eta_t(y)]
=\E_N[\eta_t(x)\eta_t(y)]- v_t^N(x) v_t^N(y)
\end{equation} 
for any $(x,y)\in V_N^{(2)}$ where we set 
\begin{equation}
\label{eq:2d_torus_minus_diagonal_definition}
V_N^{(2)}
= \{ (x,y)\in\mathbb T_N^2 \, :\, 
x \not\equiv y \, (\text{mod } N) \} .
\end{equation}
Above the notation $\overline\eta_t(x)$ refers to the centered random variable, i.e. $\overline\eta_t(x)\coloneqq \eta_{t}(x)- v_t^N(x)$ in the above case. 
We present now the conditions that we need to impose on the sequence of initial measures $\mu_N$. 

\begin{assumption}
\label{assump:initial_measure}
There exist a positive constant $C>0$ such that 
\begin{equation}
\label{eq:assum_H-1norm}
\limsup_{N\to\infty}\mathbb E_N
\big[\|\eta^2_0\|^2_{-1,N}\big]
\le C,
\end{equation}
\begin{equation}
\label{eq:assumption_initial_measure_volume}
\begin{aligned}
\limsup_{N\to\infty} \max_{x\in\mathbb T_N} \big\{ 
|v^N_0(x)| \vee |\nabla^N v^N_0(x)|
\vee | \Delta_N^{(1)} v^N_0(x)|
\big\} 
\le C,
\end{aligned}
\end{equation}
\begin{equation}
\label{eq:assumption_initial_meaure_energy}
\limsup_{N\to\infty}\max_{x\in\mathbb T_N}\big\{ |e^N_0(x)|\vee 
|\nabla^N e^N_0(x)| \big\} \le C, 
\end{equation}
and 
\begin{equation}
\label{eq:assumption_initial_measure_correlation}
\limsup_{N\to\infty} \max_{(x,y)\in V_N^{(2)}} |\varphi^N_0(x,y)|
\le C/N 
\end{equation}
where in the first assumption the norm $\|\cdot\|_{-1,N}$ will be defined in \eqref{eq:h-1norm_expression}. 
\end{assumption}

\begin{example}
Here we give an example of initial measures satisfying all items in Assumption \ref{assump:initial_measure}.
Let us consider the following Gaussian product measure whose common marginal satisfies  
\begin{equation}
\label{eq:product_measure_associated_to_continuous_profile}
\mu_N(\eta(x) \le y)
= \int_{-\infty}^y \frac{1}{\sqrt{2\pi\chi_0(x/N)}}
\exp{\Big(-\frac{(z-v_0(x/N))^2}{2\chi_0(x/N)}\Big)}  dz  
\end{equation}
for each $x\in\mathbb T_N$ and $y\in\mathbb R$ where $\chi_0(\cdot)=e_0(\cdot)-v_0(\cdot)^2$. 
Above, $v_0:\mathbb T\to\mathbb T$ and $e_0:\mathbb T\to (0,\infty)$ are bounded measurable functions such that $\chi_0(u)> 0$ for any $u\in\mathbb T$, to make the measure $\mu_N$ well-defined. 
Moreover, if we impose further that $v_0,e_0\in C^2_b(\mathbb T)$, then we can easily check that all items in Assumption \ref{assump:initial_measure}
are satisfied.
\end{example}

\subsubsection{Hydrodynamic limit}
Given a configuration $\eta\in\Omega_N$, we associate to it the empirical measure for each conservation law,  i.e.  the empirical measures of volume and energy defined, respectively, by
\begin{equation*}
\pi^{v,N}_t(du)
= \frac{1}{N} \sum_{x\in\mathbb T_N} 
\eta_t(x) \delta_{x/N} (du)
\quad \textrm{and}\quad 
\pi^{e,N}_t(du)
= \frac{1}{N} \sum_{x\in\mathbb T_N} 
\eta_t(x)^2 \delta_{x/N} (du)
\end{equation*}
where $u\in\mathbb T$ and $\delta_u$ stands for the Dirac measure at $u$.
For any continuous function $G:\mathbb T\to\mathbb R$, let us denote its integral with respect to the empirical measure by $\langle\pi^{\sigma,N}_t,G\rangle$ for each $\sigma\in \{v,e\}$. 
Given the fact that the volume element $\eta$ takes values in $\mathbb R$, the corresponding empirical measure takes values in general on the space of signed measures. Since a good tightness criteria in the space of unbounded signed measures is lacking, we are forced to discuss the convergence of these measures in a different topology. This is also the case, for example, for the empirical currents in some particle systems as mentioned in \cite{bertini2023concurrent} for the case of the exclusion process. Here let us interpret not only the volume empirical measure, but also the energy empirical measure as processes taking values in Sobolev spaces, which we recall as follows.

For each integer $z$, let a function $h_z$ be defined by 
\begin{equation}
\label{eq:sobolev_CONS_definition}
h_z(u)
= \begin{cases}
\begin{aligned}
& \sqrt 2\cos(2\pi zu) &&\text{ if } z>0, \\
& \sqrt 2\sin(2\pi zu) &&\text{ if } z<0, \\
& 1 &&\text{ if } z=0.
\end{aligned}
\end{cases}
\end{equation}
Then the set $\{h_{z},z\in\mathbb Z\}$ is an orthonormal basis of $L^{2}(\mathbb T)$. 
Consider in $L^{2}(\mathbb{T})$ the operator $1-\Delta$, which is linear, symmetric and positive. A simple computation shows that
$(1-\Delta) h_{z}=\gamma_{z}h_{z}$ where $\gamma_{z}=1+4\pi^2 z^2$.
For an integer $m\geq{0}$, denote by $H^m(\mathbb T)$ the Hilbert space induced by $C^\infty(\mathbb T)$ and the scalar product $\langle\cdot,\cdot\rangle_{m}$ defined by $\langle f,g\rangle_{m}=\langle f,(1-\Delta)^{m}g\rangle$, where recall that $\langle\cdot,\cdot\rangle$ denotes the inner product of
$L^{2}(\mathbb{T})$ and denote by $H^{-m}(\mathbb T)$ the dual of $H^m(\mathbb T)$ relatively to this inner product $\langle\cdot,\cdot\rangle_m$. 
Denote by $D(\mathbb R_+,H^{-m}(\mathbb T))$  the space of $H^{-m}(\mathbb T)$-valued functions, which are right-continuous and  with left limits, endowed with the Skorohod topology. 
Now, let us denote by $Q^{v,N}_m$(resp. $Q^{e,N}_m$) the measure on $D([0,T],H^{-m}(\mathbb T))$ associated with the empirical measure $\pi^{v,N}$ (resp. $\pi^{e,N}$) and the initial measure, i.e., 
\begin{equation}
\label{eq:sobolev_empirical_measure_definition}
Q^{v,N}_m=\mathbb P_N\circ(\pi^{v,N})^{-1} \quad\text{and} \quad
Q^{e,N}_m=\mathbb P_N\circ(\pi^{e,N})^{-1}. 
\end{equation}
We will prove in \cref{sec:hydrodynamics} that the empirical measure $\pi_t^{v,N}$ (resp. $\pi_t^{e,N}$) converges, in a sense to be precised later, to a deterministic measure $\pi_t^v$ (resp. $ \pi_t^e$).
Then, we will see that the limiting measures are absolutely continuous with respect to the Lebesgue measure, i.e., $\pi_t^{v}(du)=v(t,u)du$ and $ \pi_t^{e}(du)=e(t,u)du$, and the densities evolve according to a system of hydrodynamic equations.
It turns out that $\kappa=1$ is critical where the limiting equations are coupled in volume and energy: 
\begin{equation}
\label{eq:hdl}
\begin{cases}
\begin{aligned}
& \partial_tv
= \Delta v
+ 2\alpha \nabla v \\
& \partial_t e =\Delta e + 2 \alpha\nabla v^2 \\
& v(0,\cdot)= v_0(\cdot), \quad e(0,\cdot)=e_0(\cdot)
\end{aligned}
\end{cases}
\end{equation}
with some initial profile $v_0$ and $e_0$, where recall from \eqref{eq:weak_asymmetry} that $\alpha$ is the parameter controlling the intensity of the Hamiltonian part.
On the other hand, when $\kappa>1$, the limiting equations are no longer coupled and become the classical heat equations, which are obtained by setting $\alpha=0$ in \eqref{eq:hdl}.
We introduce now the precise sense in which we refer to as a solution to the previous equations.

\begin{definition}
Fix a pair of bounded measurable functions $v_0,e_0:\mathbb T\to \mathbb R$. 
We say that a pair of measurable functions $(v,e)=(v(t,u), e(t,u)) \in L^2([0,T]\times \mathbb T)^2$ is a weak solution of \eqref{eq:hdl} if for every $G\in C^{2}(\mathbb T)$, it holds that 
\begin{equation*}
\begin{aligned}
\int_{\mathbb T}v(t,u)G(u)du - \int_{\mathbb T}v_0(u) G(u)du 
+ \int_0^t\int_{\mathbb T}v(s,u) 
\big(\Delta G(u) - 2\alpha\nabla G(u)\big) du ds = 0
\end{aligned}
\end{equation*}
and 
\begin{equation*}
\begin{aligned}
& \int_{\mathbb T}e(t,u)G(u)du - \int_{\mathbb T}e_0(u) G(u)du \\
&\quad
+ \int_0^t\int_{\mathbb T}e(s,u) 
\Delta G(u) du ds 
- \int_0^t\int_{\mathbb T} v(s,u)^2 \nabla G(u) du ds = 0.
\end{aligned}
\end{equation*}
\end{definition}

It follows from standard arguments that the system \eqref{eq:hdl} has a unique weak solution. 
Indeed, we may notice that the first equation for the volume is decoupled from the other equation for the energy, and that one is  the classical heat equation with a drift term, so that we can easily show the uniqueness of weak solutions. 
Then, the second equation for the energy can be viewed as decoupled where the nonlinear term $v^2$ is considered to be given, and thus the uniqueness of the second equation is trivial as well.    

Next, let us recall the notion of association to a profile for the empirical measures. 

\begin{definition}
Let $\{\mu_N\}_{N\in\mathbb N}$ be a sequence of probability measures in $\Omega_N$ and let $v_0$ and $e_0$ be measurable real-valued functions defined on $\mathbb T$.
We say the sequence $\{\mu_N\}_{N\in\mathbb N}$ is associated to $(v_0,e_0)$ if for every real-valued continuous function $G$ on $\mathbb T$ and for every $\delta>0$ we have that 
\begin{equation*}
\lim_{N\to\infty}\mu_N\bigg( \eta\in\Omega_N \,;\,  
\Big| \frac{1}{N} \sum_{x\in\mathbb T_N} \eta(x) G\Big(\frac xN\Big) - \int_{\mathbb{T}}G(u)v_0(u)du \Big| > \delta\bigg) =0
\end{equation*}
and 
\begin{equation*}
\lim_{N\to\infty}\mu_N\bigg( \eta\in\Omega_N \,;\,  
\Big| \frac{1}{N} \sum_{x\in\mathbb T_N} \eta(x)^2 G\Big(\frac xN \Big) - \int_{\mathbb{T}}G(u)e_0(u)du \Big| > \delta\bigg) =0.
\end{equation*}
\end{definition}

Now, let us state our main theorem for the hydrodynamic limit. 

\begin{theorem}
\label{thm:HDL weak regime} 
Let $v_0:\mathbb T\to \mathbb R$ and $e_0:\mathbb T\to (0,\infty)$ be bounded measurable functions.
Let $\{\mu_N\}_{N\in\mathbb N}$ be a sequence of probability measures on $\Omega_N$ associated to the pair $(v_0,e_0)$, and assume that $\{\mu_N\}_{N\in\mathbb N}$ satisfies Assumption \ref{assump:initial_measure}.  
Moreover, assume $\kappa>1$ where $\kappa$ is the constant appearing in \eqref{eq:weak_asymmetry}.  
Then, for every $t\in [0,T]$, the sequence of empirical measures $\{(\pi^{v,N}_t, \pi^{e,N}_t)\}_{N\in\mathbb N}$ converges in probability to the absolute continuous measure $(v(t,u)du, e(t,u)du)$, that is, for any $\delta>0$ and for any $G\in C^2(\mathbb{T})$ we have 
\begin{equation*}
\lim_{N\to\infty}
Q^{\sigma,N}_m 
\bigg(\Big|\langle \pi^{\sigma,N}_{t},G\rangle-\int_{\mathbb T} \sigma(t,u)G(u)du \Big|>\delta\bigg) =0 
\end{equation*}
for each $\sigma\in \{e,v\}$ provided $m>5/2$, where recall that the measure $Q^{v,N}_m$ and $Q^{e,N}_m$ were defined for processes taking values on Sobolev space $H^{-m}(\mathbb T)$, see \eqref{eq:sobolev_empirical_measure_definition}.  
Moreover, the pair $(v(t,u), e(t,u))$ is the unique weal solution of \eqref{eq:hdl} under $\alpha=0$ with initial profile $(v_0,e_0)$.  
\end{theorem}

\begin{remark}
Note that the energy empirical measure could be interpret as a measure-valued process, so that we can state the above theorem without imposing any condition on the regularity $m$. 
Here we decided to work with Sobolev space valued processes also for the energy empirical measure, just for the sake of consistency. 
\end{remark}

\begin{theorem}
\label{thm:HDL_volume}
Let $v_0$ be a bounded real-valued measurable function on $\mathbb T$ and let $\{\mu_N\}_{N\in\mathbb N}$ be a sequence of probability measures on $\Omega_N$ associated to $v_0$ and assume that $\{\mu_N\}_{N\in\mathbb N}$ satisfies Assumption \ref{assump:initial_measure}.  
Moreover, assume that the parameter $\alpha_N$ is taken under the weak asymmetry \eqref{eq:weak_asymmetry} with $\kappa=1$.
Then, for every $t\in [0,T]$, the sequence of empirical measures $\{\pi^{v,N}_t\}_{N\in\mathbb N}$ converges in probability to the absolute continuous measure $v(t,u)du$, that is, for any $\delta>0$ and for any $G\in C^2(\mathbb{T})$ we have 
\begin{equation*}
\lim_{N\to\infty}Q^{v,N}_m \bigg(\Big|\langle \pi^{v,N}_{t},G\rangle-\int_{\mathbb T} v(t,u)G(u)du \Big|>\delta\bigg) =0
\end{equation*}
provided $m>5/2$, where $v(t,u)$ is the unique weak solution of the first equation in \eqref{eq:hdl}. 
\end{theorem}

Meanwhile, for the energy, the convergence in probability as stated above for the volume is left open in this paper. 
Instead, we are forced to state the result in mean as follows, but with some quantitative estimate on the rate of convergence.

\begin{theorem}
\label{thm:HDL_energy}
Assume that $ v_0, e_0\in C^\infty_b(\mathbb T)$ and that $e_0$ is positive. 
Let $\kappa=1$ and let $\{ ( v^N_t(x),e^N_t(x))\}_{t\ge 0, x\in\mathbb T_N}$ (resp. $\{ ( v(t,u), e(t,u))\}_{t\ge0 , u\in \mathbb T}$) be the unique solution of \eqref{eq:discrete_hdl} (resp. \eqref{eq:hdl}) with the following initial conditions: 
\begin{equation}
\label{eq:initial_cond}
v^N_0(x) =  v_0(x/N), \quad
e^N_0(x)=e_0(x/N)
\end{equation}
and 
\begin{equation*}
v(0,u) =  v_0(u),\quad
e(0,u) = e_0(u). 
\end{equation*}
Let $\mu_N$ be a product measure defined by \eqref{eq:product_measure_associated_to_continuous_profile}, and assume further that the sequence of measures $\{\mu_N\}_{N\in\mathbb N}$ satisfies \eqref{eq:assum_H-1norm} and \eqref{eq:assumption_initial_measure_correlation}.
Then, there exists $C=C(T,v_0,e_0)>0$ such that 
\begin{equation}
\label{eq:convergece_profile_energy}
\sup_{0\leq t\leq T}
\max_{x\in\mathbb T_N} 
\big| e^N_t(x)-e_t(x/N)\big| 
\le C\frac{\log N}{N}.
\end{equation}
\end{theorem}

\begin{remark}
We note that the assumption $ v_0, e_0\in C^\infty_b(\mathbb T)$ could be relaxed to ask less regularity and also the condition \eqref{eq:initial_cond} by asking the same bound as in \eqref{eq:convergece_profile_energy} for the initial time. 
\end{remark}

\begin{remark}
\label{rmk:weak convergence energy}
As a consequence of \eqref{eq:convergece_profile_energy},  the sequence $\{\pi^{e,N}_t\}_{N\in\mathbb N}$ converges in mean in the following sense: 
\begin{equation*}
\lim_{N\to\infty}\mathbb{E}_N \big[\langle\pi_t^{e,N},G\rangle\big] 
=\int_\mathbb T e(t,u)G(u)du
\end{equation*}
for any $G\in C^2(\mathbb T)$. 
This will be proved in the end of \cref{sec:hydrodynamics}.
\end{remark}

\subsubsection{Non-equilibrium fluctuations}
Recall that we defined the discrete profiles for volume and energy,  $ v^N_t(x)$ and $e^N_t(x)$, by \eqref{eq:element_definition}. 
Now, we are in a position to state our main theorem for the non-equilibrium fluctuations for the volume. 
Let us define the  fluctuation field associated to the volume by 
\begin{equation*}
\mathcal{V}^N_t(G)
= \frac{1}{\sqrt{N}}
\sum_{x\in\mathbb T_N} 
\big(\eta_t(x) -  v^N_t(x)\big) 
T^-_{f_Nt}G \Big( \frac{x}{N}\Big)
\end{equation*}
for any test function $G\in C(\mathbb T)$, where $f_N$ is a moving frame defined by
\begin{equation}
\label{eq:moving_frame_definition}
f_N =  2\alpha_N N^2
\end{equation}
and $T^\pm_\cdot$ is a canonical shift defined by $T^\pm_zg(x)= g(x\pm z)$ for any $z\in\mathbb R$. 
Let $\mathcal Q^N_m$ denote the probability measure on $D([0,T],H^{-m}(\mathbb T))$ induced by the volume fluctuation field $\mathcal{V}^N_\cdot$ starting from an initial measure $\mu_N$, which will be specified in the statement.  
Then, we have the following statement for the non-equilibrium fluctuations for the volume field.

\begin{theorem}
\label{thm:main fluctuations}
Let $m>5/2$ be fixed. 
Assume that the sequence of initial measures $\{\mu_N\}_{N\in\mathbb N}$ satisfies Assumption \ref{assump:initial_measure}. 
Moreover, assume that bounded measurable functions $v_0:\mathbb T\to\mathbb R$ and $e_0:\mathbb T\to(0,\infty)$ satisfies $e_0(u)-v_0(u)^2\ge 0$ for any $u\in\mathbb T$.
Finally, assume that the sequence of initial density fields $\mathcal{V}_0^N$ converges in distribution to an $H^{-m}(\mathbb T)$-valued random variable $\mathcal{V}_0$ as $N\to\infty$.  
\begin{itemize}
\item[a)] 
Let $\mathcal{M}^N_t$ be a mean-zero martingale associated naturally to the volume fluctuation fields $\mathcal V^N_\cdot$ (see \eqref{eq:dynkin_martingale_definition} for the definition).  
The sequence $\{\mathcal{M}^N_t: t\in [0,T]\}_{N\in\mathbb N}$ is tight in $D([0,T],H^{-m}(\mathbb T))$, so that there exists a subsequence of $N$, which is again denoted by the same letter by abuse of notation, along which the sequence converges to a limit $\{\mathcal{M}_t ; t\in [0,T]\}$ in distribution:
\begin{equation*}
\lim_{N\to\infty} \mathcal{M}^N_\cdot = \mathcal{M}_\cdot 
\end{equation*}
where the limit $\mathcal{M}_\cdot $ is again a mean-zero martingale.
\item [b)]For any $f \in C^\infty(\mathbb T)$, we have that
\begin{equation}\label{eq:lim_qv} 
\lim_{N\to \infty} \mathbb E_N[\langle \mathcal{M}^N(f)\rangle_t]  
= \int_0^t \int_{\mathbb T} 2 \chi(s,u) (\nabla f(u))^2 du ds
\end{equation}
where $\chi(t,u) = e(t,u) - v(t,u)^2$, and $(v,e)$ is the unique weak solution of \eqref{eq:hdl} with initial profile $(v_0,e_0)$.  
\item  [c)]
The sequence $\{\mathcal Q^N_m\}_N$ is tight and any of its limit points $\mathcal Q_m$ satisfies the identity 
\begin{equation}
\label{eq:volume_limit_field}
\mathcal{V}_t(f) 
= \mathcal{V}_0(\mathcal T_tf) 
+ \mathcal{M}_t(f)
\end{equation}
for any $f\in \mathcal{D}(\mathbb T)$, where $\mathcal{M}_t$ is a mean-zero  martingale with continuous trajectories and $\mathcal T_t$ is the semigroup generated by the one-dimensional Laplacian $\Delta$. Moreover, $\mathcal V_0$ and $\mathcal M_t$ are uncorrelated in the sense that, for all $f,g\in\mathcal D(\mathbb T)$ and for all $t\in[0,T]$, it holds $E_{\mathcal Q_m}[\mathcal  V_0(g)\mathcal M_t(f)]=0$.
\item [d)] If $\mu_N$ is the invariant measure $\mu_{\rho,\beta}$ for some $\rho\in\mathbb R$ and $\beta>0$, then we can strength our result in the sense that we have the convergence of the sequence $\{\mathcal Q^N_m\}_N$ to the solution of the Ornstein-Uhlenbeck equation
$$
d\mathcal V_t=\Delta \mathcal V_t dt+\sqrt{\frac{2}{\beta}} \nabla \dot{\mathcal W_t} 
$$
again provided that $m>5/2$, where $\dot{\mathcal W_t}$ is a one-dimensional space-time white noise. 
\end{itemize}
\end{theorem}

\begin{remark}
Regarding  item b) in \cref{thm:main fluctuations}, it is not difficult to show that the uniform non-negativity of the static compressibility $\chi$ at the initial time propagates to any time $t\ge 0$. 
Indeed, notice that the time evolution of $\chi$ is given by
\begin{equation*}
\partial_t \chi 
= \Delta\chi 
+ 2(\nabla v)^2.
\end{equation*}
Then, due to the maximum principle, we have $\chi(t,u) \ge 0$ for any $t\ge 0$ and for any $u\in\mathbb T$.  
\end{remark}

We observe that in the previous result (apart from item d)) we do not have a full convergence result because we do not have enough information about the quadratic variation of the limiting martingale $\mathcal M_t$. We were not able to prove convergence in law to its expectation, but only that this convergence holds in expectation, and this does not fully characterize the distribution of the limit. One way to prove indeed that the convergence in distribution holds is to obtain four-point correlation bounds. However, this seems much more challenging than deriving bounds on the two-point correlations (as we do below). The reason for this difficulty comes from the presence of the Hamiltonian dynamics, since we obtain an evolution equation that is written in terms of non-homogeneous random walks for which sharp bounds on the transition probabilities are not available in the literature. Therefore, we leave this as an open problem. 
\medskip

Now we give some comments on the strategy of the proof. First we observe that our model deals with unbounded variables which do not have a sign, this causes troubles when controlling many additive functionals that appear along the way. 
Moreover, we do not have any information about higher moments of the volume; this has to be extracted from correlation estimates, in what concerns second-moment bounds or the so-called $\mathcal H^{-1,N}$ estimates for averages of time-integrated fourth moments. 
Our main results regarding moment bounds can be stated as follows. The first result gives good control on the averages of the fourth moments.
\begin{theorem}
\label{thm:fourth-moment_bound}
Let $T>0$ and assume that the sequence of probability measures $\{\mu_N\}_N$ satisfies
\eqref{eq:assum_H-1norm}.
Then, there exists some $C=C(T)>0$ such that for any $t\in[0,T]$ it holds 
\begin{equation*}
\mathbb{E}_N \bigg[\int_0^t\frac{1}{N}\sum_{x\in\mathbb{T}_N}\eta_s(x)^4ds\bigg] 
\le C(1 + \alpha_N N).   
\end{equation*}
\end{theorem}

The second result gives a good control of the two-point correlation function.
\begin{theorem}
\label{lem:correlation_estimate}
Assume $\kappa\ge 1$, and assume that the sequence of initial measures $\{\mu_N\}_{N\in\mathbb N}$ satisfies the conditions \eqref{eq:assumption_initial_measure_volume} and \eqref{eq:assumption_initial_measure_correlation}. Then, there exists a constant $C=C( v_0,e_0,T)>0$ that does not depend on $N$, such that 
\begin{equation}
\label{eq:volume_correlation_estimate}
\sup_{0\le t\le T} 
\| \varphi^{N}_s\|_{\ell^\infty(V^{(2)}_N)}
\le C/N. 
\end{equation}
\end{theorem}

Here follows an outline of the paper. In \cref{sec:hydrodynamics} we present the proof of the hydrodynamic limit. 
In \cref{sec:fluctuations} we present the proof of the non-equilibrium fluctuations for the volume. 
Both results substantially rely on moment bounds and decay of the two-point correlation function. These are presented in Sections \ref{sec:moment bound} and \ref{sec:two-point_correlation_estimate}.

\section{Hydrodynamic limit}
\label{sec:hydrodynamics}
\subsection{Proof outline}
In this section we prove Theorems \ref{thm:HDL weak regime}, \ref{thm:HDL_volume} and \ref{thm:HDL_energy}. 
To give an intuition,  we note that it is possible to check that the discrete profiles for the volume and energy satisfy 
\begin{equation}
\label{eq:discrete_hdl}
\begin{cases}
\begin{aligned}
& \partial_t  v^N_t(x)
= \Delta^N  v^N_t(x) 
+ 2\alpha_NN \widetilde{\nabla}^N  v^N_t(x),\\
&\partial_t e^N_t(x)
= \Delta^N e^N_t(x)
 + 2\alpha_NN 
\nabla^N \big( v^N_t(x-1) v^N_t(x)\big)  ,  
\end{aligned}
\end{cases}
\end{equation}
which is a discrete version of the hydrodynamic equation \eqref{eq:hdl}.  
Above, recall that $\Delta^N$ and $\widetilde{\nabla}^N$ denote the discrete differential operators defined in \eqref{eq:discrete_operators_definition}.

The strategy of our proof uses the  entropy method
\cite{guo1988nonlinear} (see also \cite[Chapters 4 and 5]{kipnis1999scaling} for a pedagogical review), and therefore we will proceed in three steps showing tightness of the sequences of probability measures induced by the empirical measures, then characterize the limit points and finally from the uniqueness of the limit we will conclude the convergence result. 
To that end, from  ~\cite[Lemma A1.5.1]{kipnis1999scaling}, for any $G\in C^2(\mathbb{T})$, 
\begin{equation*}
M_t^{v,N}(G)=\langle \pi^{v,N}_t,G\rangle -\langle \pi^{v,N}_0,G\rangle-\int_0^t L_N\langle \pi^{v,N}_s,G\rangle ds
\end{equation*} 
and $M^{v,N}_t(G)^2 - \langle M^{v,N}(G)\rangle_t$ with 
\begin{equation*}
\langle M^{v,N}(G)\rangle_t
= \int_0^t 
\big( L_N\langle \pi^{v,N}_s,G\rangle^2 
-2\langle \pi^{v,N}_s,G\rangle 
L_N \langle \pi^{v,N}_s,G\rangle\big) ds 
\end{equation*}
are martingales -the Dynkin's martingales -  with respect to the natural filtration of the process. 
A direct computation shows  that
\begin{equation}
\label{eq:dynkin_volume}
\begin{aligned}
M^{v,N}_t(G)
&=\langle \pi^{v,N}_t,G\rangle-\langle \pi^{v,N}_0,G\rangle
-\int_0^t\frac{1}{N}\sum_{x\in\mathbb{T}}
\eta_s(x)\Delta^N G(x/N) ds \\
&\quad +2\alpha_N\int_0^t\sum_{x\in\mathbb{T}}
\eta_s(x) \widetilde{\nabla}^N G(x/N)ds 
\end{aligned}
\end{equation}
and 
\begin{equation*}
\langle M^{v,N}(G)\rangle_t
=\int_0^t \frac{1}{N^2}
\sum_{x\in\mathbb{T}_N}\big(\eta_s(x+1)-\eta_s(x)\big)^2\big(\nabla^N G(x/N)\big)^2ds.
\end{equation*}
Similarly for the energy, we have the following expression for the Dynkin martingale, as well as its quadratic variation for the energy empirical measure:
\begin{equation}
\label{eq:dynkin_energy}
\begin{split}
M^{e,N}_t(G)
&=\langle \pi^{e,N}_t,G\rangle - \langle \pi^{e,N}_0,G\rangle -\int_0^t\frac{1}{N}\sum_{x\in\mathbb{T}_N}\eta_s(x)^2 \Delta^NG(x/N)ds \\
&\quad +2\alpha_N \int_0^t \sum_{x\in\mathbb{T}_N}
\eta_s(x+1)\eta_s(x)\nabla^NG(x/N) ds 
\end{split}
\end{equation}
and 
\begin{equation}
\label{eq:QV energy}
\langle M^{e,N}(G)\rangle_t
=\int_0^t \frac{1}{N^2} \sum_{x\in\mathbb{T}_N}\big(\eta^2_s(x+1)-\eta^2_s(x)\big)^2
(\nabla^N G(x/N))^2 ds.
\end{equation}
To show \cref{thm:HDL weak regime} and \cref{thm:HDL_volume}, we will rely on the above Dynkin's martingale decompositions. 
To that end, we will show that each term in the decomposition is tight in \cref{subsec:HDL_tightness} and then characterize the limiting points in \cref{subsec:clp}. 
The proof of \cref{thm:HDL_energy} will be separately given in \cref{subsec:HDL_quantitative_energy}.

\subsection{Tightness}
\label{subsec:HDL_tightness}
\subsubsection{Volume empirical measures}
Here we notice that, since volume takes values in $\mathbb R$, it is not possible to consider the empirical measure as process which takes values in a space of positive measures, as it is usually done for non-negative valued processes such as particle system (see \cite{kipnis1986central}).   
Therefore, we need to consider the empirical measure as a process which is taking values on the Sobolev space with a negative index, and then project it on each time $t$. 
We first show the result below, following \cite[Chapter 11]{kipnis1999scaling}. 

\begin{lemma}
\label{lem:tightness_volume}
Fix any $\kappa\ge 1$ and let $m>5/2$.
Then, the sequence of probability measures $\{Q^{v,N}_m\}_{N\in\mathbb N}$ is tight in $D([0,T],H^{-m}(\mathbb T))$.  
\end{lemma}

Before giving a proof of \cref{lem:tightness_volume}, let us recall some basic notion regarding the tightness of probability measures in our setting.

\begin{definition}
For $\delta>0$ and a path $\pi$ in ${D([0,T],H^{-m}(\mathbb T))}$, the uniform modulus of continuity of $\pi$, is defined by
\begin{equation*}
\omega_{\delta}(\pi)=\sup_{\substack{|s-t|<\delta, \\0\leq{s,t}\leq{T}}}\|\pi_{t}-\pi_{s}\|_{-m}.
\end{equation*}
\end{definition}
The first result gives sufficient conditions for a subset to be weakly relatively compact.
\begin{lemma}
A subset $A$ of $D([0,T],H^{-m}(\mathbb T))$ is relatively compact for the uniform weak topology if the next two conditions hold:
\begin{equation*}
\sup_{Y\in{A}}\sup_{0\leq{t}\leq{T}}\|\pi_{t}\|_{-m}<\infty
\quad \textrm{and}\quad \lim_{\delta\rightarrow{0}}\sup_{\pi\in{A}}\omega_{\delta}(\pi)=0.
\end{equation*}
\end{lemma}
From this lemma we can transform the concept of tightness by replacing for compactness its Arzel\`{a}-Ascoli characterization.
Thus, we get a criterion for tightness of a sequence of probability measures defined on $D([0,T],H^{-m}(\mathbb T))$.
\begin{lemma}
A sequence $\{P_{N}, N\geq{1}\}$ of probability measures defined on $D([0,T],H^{-m}(\mathbb T))$ is tight if the following two conditions hold:
\begin{equation*}
    \begin{split}
       & (\text{a}) \lim_{A\rightarrow{\infty}}\limsup_{N\rightarrow{\infty}}P_{N}\Big(\sup_{0\leq{t}\leq{T}}\|\pi_{t}\|_{-m}>A\Big)=0\\
       &(\text{b})\lim_{\delta\rightarrow{0}}\limsup_{N\rightarrow{\infty}}P_{N}\big(\omega_{\delta}(\pi)\geq{\varepsilon}\big)=0,\;\text{
for every}\; \varepsilon>0.
    \end{split}
\end{equation*}
\end{lemma}

Now, we show the tightness of the measures associated to the volume. 
In order to show that the sequence $\{Q^{v,N}_m\}_{N}$ is tight, we have to check that the next two limits hold
\begin{equation*}
\begin{split}
&\lim_{A\rightarrow{+\infty}}\limsup_{N\rightarrow{+\infty}}Q^{v,N}_m\Big(\sup_{0\leq{t}\leq{T}}\|\pi^v_{t}\|_{-m}^{2} > A \Big) =0, \\&
\lim_{\delta\rightarrow{0}}\limsup_{N\rightarrow{+\infty}}Q^{v,N}_m \big(\omega_{\delta}(\pi^v)\geq\varepsilon \big)=0
\end{split}
\end{equation*}
for any $\varepsilon>0$.

\begin{lemma} 
\label{lem:HDL_tighness_key_estimate}
There exists some $C=C(T)>0$ such that for every $z\in\mathbb Z$, 
\begin{equation*}
{\limsup
_{N\rightarrow{+\infty}}}\,\,
\mathbb{E}_{N}\Big[\sup_{0\leq{t}\leq{T}}|\langle\pi^{v,N}_{t},h_{z}\rangle|^{2}\Big] 
\le C(\alpha^2 z^2\mathbf{1}_{\kappa=1}
+ z^4).
\end{equation*}
\end{lemma}
\begin{proof}
To prove the lemma, we estimate separately each term in the martingale decomposition \eqref{eq:dynkin_volume} with $G=h_z$. 
First, note that a simple computation together with \eqref{eq:assumption_initial_measure_correlation}, shows that
\begin{equation*}
\lim_{N\rightarrow{+\infty}}\mathbb{E}_{N}
\big[\langle \pi^{v,N}_{0},h_{z}\rangle^2 \big]=0.
\end{equation*}
The contribution of the martingale also vanishes, by combining Doob's inequality with the assumption \eqref{eq:assumption_initial_meaure_energy}.   
Now, we analyze the integral terms. 
First, we start with the contribution from the symmetric part. 
To that end note that 
\begin{equation*}\begin{split}
&\mathbb E_N\Big[\Big(\sup_{0\leq{t}\leq{T}}\int_0^t\frac{1}{N}\sum_{x\in\mathbb{T}}
\eta_s(x)\Delta^N h_z(x/N) ds \Big)^2\Big] \\
&\quad\le T \mathbb E_N\Big[\int_0^T
\frac{1}{N^2}\sum_{x,y\in\mathbb{T}_N}\eta_s(x)\eta_s(y)\Delta^N h_z(x/N)\Delta^N h_z(y/N)ds \Big]
\\
&\quad\leq \frac{T}{N^2} \sum_{x,y\in\mathbb{T}_N}\int_0^T\varphi_s^N(x,y)\Delta^N h_z(x/N)\Delta^N h_z(y/N)ds
\\
&\qquad+\frac{T}{N^2}\sum_{x,y\in\mathbb{T}_N}\int_0^T v^N_s(y) v_s^N(x) \Delta^N h_z(x/N)\Delta^N h_z(y/N)ds
\\&\quad 
\le C(N^{-1}+1) z^4
\end{split}
\end{equation*}
for some $C=C(T)>0$ {where we used the correlation estimate (\cref{lem:correlation_estimate})}. 

For the anti-symmetric part we reason as follows:
\begin{equation*}\begin{split}
&\mathbb E_N\Big[\Big(\sup_{0\leq{t}\leq{T}}\int_0^t\alpha_N 
\sum_{x\in\mathbb{T}}\eta_s(x)\widetilde{\nabla}^N h_z(x/N) ds \Big)^2\Big] \\
&\quad\le T \mathbb E_N\Big[\int_0^T\alpha_N^2\sum_{x,y\in\mathbb{T}_N}\eta_s(x)\eta_s(y)\widetilde{\nabla}^N h_z(x/N) \widetilde{\nabla}^N h_z(x/N) ds\Big]
\\
&\quad\leq T \alpha_N^2\sum_{x,y\in\mathbb{T}_N}\int_0^T\varphi_s^N(x,y)\widetilde{\nabla}^N h_z(x/N) \widetilde{\nabla}^N h_z(x/N)ds
\\
&\qquad+T\alpha_N^2\sum_{x,y\in\mathbb{T}_N}\int_0^T v^N_s(y) v_s^N(x) \widetilde{\nabla}^N h_z(x/N) \widetilde{\nabla}^N h_z(x/N)ds\\
&\quad\le C \alpha_N^2(N+ N^2) z^2
\end{split}
\end{equation*}
for some $C=C(T)>0$. 
Hence, we complete the proof of the desired estimates. 
\end{proof}

\begin{corollary}
Assume $m>5/2$ and $\kappa\ge 1$. Then, it holds:
\begin{equation*}\begin{split}
&(1)\qquad\limsup_{N\rightarrow{+\infty}}\mathbb{E}_{N}\Big[\sup_{0\leq{t}\leq{T}}\|\pi_{t}^{v,N}\|_{-m}^{2}\Big] < \infty 
\\& (2)\qquad 
\lim_{n\rightarrow{+\infty}}\limsup_{N\rightarrow{+\infty}}\mathbb{E}_{N}\Big[\sup_{0\leq{t}\leq{T}}\sum_{|z|\geq{n}}(\langle \pi_{t}^{v,N},h_{z}\rangle)^{2}\gamma_{z}^{-m}\Big]=0.
\end{split}\end{equation*}
\end{corollary}
\begin{proof}
First, let us show the first item. 
Since
\begin{equation*}
\limsup_{N\rightarrow{+\infty}}\mathbb{E}_{N}\Big[\sup_{0\leq{t}\leq{T}}\|\pi^{v,N}_{t}\|_{-m}^{2}\Big]
\le \limsup_{N\rightarrow{+\infty}}
\sum_{z\in{\mathbb{Z}}}\gamma_{z}^{-m} \mathbb{E}_{N}\Big[\sup_{0\leq{t}\leq{T}}\langle \pi^{v,N}_{t},h_{z}\rangle ^{2}\Big] ,
\end{equation*}
and from \cref{lem:HDL_tighness_key_estimate} the last display is bounded  as long as $m>5/2$ and $\kappa\ge 1$.
The proof of the second assertion is analogous and we omit the proof here. 
\end{proof}

Note that condition (1) above holds as a consequence of item (a) of the previous corollary. It remains now to prove (2) but this follows from the next lemma:
\begin{lemma}
For every $n\in{\mathbb{N}}$ and every $\varepsilon>{0}$,
\begin{equation*}
\lim_{\delta\rightarrow{0}}
\limsup_{N\rightarrow{+\infty}}
\mathbb{P}_N
\bigg(\sup_{\substack{|s-t|<\delta, \\0\leq{s,t}\leq{T}}} \,
\sum_{|z|\leq{n}}(\langle \pi^{v,N}_{t}-\pi^{v,N}_{s},h_{z}\rangle)^{2}\gamma_{z}^{-m}>\varepsilon\bigg)=0 . 
\end{equation*}
\end{lemma}
\begin{proof}
The lemma follows from showing that
\begin{equation*}
\lim_{\delta\rightarrow{0}}
\limsup_{N\rightarrow{+\infty}}
\mathbb{P}_N \bigg(\sup_{\substack{|s-t|<\delta,\\0\leq{s,t}\leq{T}}}\, 
(\langle \pi^{v,N}_{t}-\pi^{v,N}_{s},h_{z}\rangle)^{2}>\varepsilon\bigg)=0
\end{equation*}
for every $z\in{\mathbb{Z}}$ and $\varepsilon>0$.
Recalling \eqref{eq:dynkin_volume} it follows from the next
claim.
For any function $f\in{C^\infty(\mathbb{T})}$ and  for every $\varepsilon>0$ it holds 
\begin{equation*}
\lim_{\delta\rightarrow{0}}
\limsup_{N\rightarrow{+\infty}}
\mathbb{P}_N \bigg( \sup_{\substack{|s-t|<\delta,\\0\leq{s,t}\leq{T}}} \,
|M_{t}^{v,N}(f)-M_{s}^{v,N}(f)|>\varepsilon \bigg) =0.
\end{equation*}
To prove the claim we denote by $\omega'_{\delta}(M^{v,N}(f))$ the modified modulus of continuity defined as
\begin{equation*}
\omega'_{\delta}(M^{v,N}(f))=\inf_{\substack{\{t_{i}\}}}\quad\max_{\substack{0\leq{i}\leq{r}}}\quad\sup_{\substack{t_{i}\leq{s}<{t}\leq{t_{i+1}}}}|M^{v,N}_t(f)
-M^{v,N}_s(f)|
\end{equation*}
where the infimum is taken over all partitions of $[0,T]$ such that $0=t_{0}<t_{1}<...<t_{r}=T$ with $t_{i+1}-t_{i}>\delta$ for each $i=0,\ldots,r$. 

Note that 
\begin{equation*}\begin{split}
\sup_{\substack{t}}|M^{v,N}_t(f)- M_{t_{-}}^{v,N}(f)|&=\sup_{\substack{t}}|\langle \pi_{t}^{v,N},f\rangle -\langle \pi _{t_{-}}^{v,N},f\rangle |\\&\leq \frac{\|\nabla{f}\|_{\infty}}{N^{2}}\sup_t\Big|\sum_{x\in\mathbb T_N}\eta_t(x)\Big|=\frac{\|\nabla{f}\|_{\infty}}{N^{2}} \Big|\sum_{x\in\mathbb T_N}\eta_0(x)\Big|.\end{split}
\end{equation*}
In the last line we used the conservation law. Now it is enough to apply Chebyshev's inequality and use the assumption \eqref{eq:assumption_initial_measure_correlation} to show that this term does not contribute to the limit. Finally note that 
\begin{equation*}
\omega_{\delta}(M^{v,N}(f))
\le 2\omega'_{\delta}(M^{v,N}(f))+\sup_{\substack{t}}|M^{v,N}_t(f)-M_{t_{-}}^{v,N}(f)|  ,
\end{equation*}
and thus the proof ends if we show that
\begin{equation*}
\lim_{\delta\rightarrow{0}}
\limsup_{N\rightarrow{+\infty}}
\mathbb{P}_N \Big(\omega'_{\delta}(M^{v,N}(f))>\varepsilon\Big)=0 
\end{equation*}
for every $\varepsilon>0$.
In order to show the last display, by the Aldous' criterion (see \cite[Proposition 4.1.6]{kipnis1999scaling}), it is enough to show that:
\begin{equation*}
\lim_{\delta\rightarrow{0}}
\limsup_{N\rightarrow{+\infty}}
\sup_{\substack{\tau\in{\mathfrak{T}_{\tau}}\\0\leq{\theta}\leq{\delta}}}
\mathbb{P}_N \Big( |M_{\tau+\theta}^{v,N}(f)- M_{\tau}^{v,N}(f)|>\varepsilon\Big)=0 
\end{equation*}
for every $\varepsilon>0$. 
Here $\mathfrak {T}_{\tau}$ denotes the family of all stopping times, with respect to the canonical filtration, bounded by $T$. 
From Chebyshev's inequality and the Optional Sampling Theorem, we can bound
\begin{equation*}
\begin{aligned}
\mathbb{P}_N \Big( |M_{\tau+\theta}^{v,N}(f)-M_{\tau}^{v,N}(f)|>\varepsilon\Big)
& \le \frac{1}{\varepsilon^2}\mathbb{E}_{N}\Big[(M_{\tau+\theta}^{v,N}(f))^{2}-(M_{\tau}^{v,N}(f))^{2}\Big] \\
&= \frac{1}{N^2\varepsilon^2} 
\int_{\tau}^{\tau +\theta}
\sum_{x\in\mathbb T_N} 
\mathbb E_N\big[(\eta_s(x)-\eta_s(x+1))^2\big]
(\nabla^N f(\tfrac{x}{N}))^2 ds \\
&\le \frac{C\theta}{N\varepsilon^2} 
\|e^N_0\|_{\ell^\infty(\mathbb T_N)} 
\| \nabla f\|^2_{L^\infty(\mathbb T)} 
\end{aligned}
\end{equation*}
with some $C>0$. 
This ends the proof since the utmost right-hand side vanishes as $N\rightarrow{+\infty}$. 
\end{proof}

\subsubsection{Energy empirical measures}
Now let us shift to showing the tightness of the sequence of empirical measures associated to the energy in the weak asymmetric regime $\kappa>1$. 

\begin{lemma}
\label{lem:tightness_energy}
Fix any $\kappa\ge 1$ and let $m>5/2$.
Then, the sequence of probability measures $\{Q^{e,N}_m\}_{N\in\mathbb N}$ is tight in $D([0,T],H^{-m}(\mathbb T))$.  
\end{lemma}

As above, in order to show that the sequence $\{Q^{e,N}_m\}_{N}$ is tight, we have to check that 
\begin{equation*}
\begin{split}
&\lim_{A\rightarrow{+\infty}}
\limsup_{N\rightarrow{+\infty}}
Q^{e,N}_m\Big(\sup_{0\leq{t}\leq{T}}\|\pi^e_{t}\|_{-m}^{2} > A \Big) =0. \\&
\lim_{\delta\rightarrow{0}}
\limsup_{N\rightarrow{+\infty}}
Q^{e,N}_m \big(\omega_{\delta}(\pi^e)\geq\varepsilon \big)=0
\end{split}\end{equation*}
for any $\varepsilon>0$. 
\begin{lemma} 
There exists some $C=C(T)>0$ such that for every $z\in\mathbb Z$, 
\begin{equation*}
{\limsup
_{N\rightarrow{+\infty}}} \,
\mathbb{E}_{N}\Big[\sup_{0\leq{t}\leq{T}}|\langle\pi^{e,N}_{t},h_{z}\rangle|^{2}\Big] 
\le C(1+z^2\alpha^2(1+\alpha)\mathbf{1}_{\kappa=1}
+ z^4).
\end{equation*}
\end{lemma}
\begin{proof}
To prove the lemma, we estimate separately each term in the martingale decomposition \eqref{eq:dynkin_energy} with $G=h_z$. 
A simple computation shows that
\begin{equation*}
\mathbb{E}_{N}
\big[\big(
\langle \pi^{e,N}_{0},h_{z}\rangle \big)^{2} \big]
\leq \mathbb{E}_{N}
\bigg[\Big(\frac{1}{N}\sum_{x\in\mathbb T_N} \eta_0(x)^2 h_{z}(x/N) \Big)^{2} \bigg]
\leq 2\|e^N_0\|_{\ell^\infty(\mathbb T_N)} ,
\end{equation*}
which is finite by \eqref{eq:assumption_initial_meaure_energy}.  
Next, recall from \eqref{eq:QV energy} {with $G=h_z$} that 
\begin{equation*}
\langle M^{e,N}(h_z)\rangle_t
=\int_0^t \frac{1}{N^2} \sum_{x\in\mathbb{T}_N}\big(\eta^2_s(x+1)-\eta^2_s(x)\big)^2
(\nabla^N h_z(x/N))^2 ds.
\end{equation*} 
Then, we can see that the contribution of the martingale vanishes, by combining Doob's inequality with the assumption \eqref{eq:assum_H-1norm} and Theorem \ref{thm:fourth-moment_bound}.
Finally, we analyze the integral terms.  
First, we start with the contribution from the symmetric part. 
To that end, note that 
\begin{equation*}\begin{split}
&\mathbb E_N\bigg[ \Big(\sup_{0\leq{t}\leq{T}}\int_0^t\frac{1}{N}\sum_{x\in\mathbb{T}}
\eta_s(x)^2\Delta^N h_z(x/N) ds \Big)^2\bigg] \\
&\quad\le T \mathbb E_N\bigg[\int_0^T\Big(
\frac{1}{N}\sum_{x\in\mathbb{T}_N}\eta_s(x)^2\Delta^N h_z(x/N)\Big)^2ds\bigg]
\\
&\quad\leq \frac{T}{N^2}  
\int_0^T \sum_{x\in\mathbb{T}_N}
\mathbb E_N[\eta_s(x)^4] ds  
\sum_{x\in\mathbb{T}_N} (\Delta^N h_z(x/N))^2
\\&\quad
\le C (1+\alpha_N N )z^4
\end{split}
\end{equation*}
for some $C=C(T)>0$, where we used the Cauchy-Schwarz inequality and \cref{thm:fourth-moment_bound}.  
For the anti-symmetric part we reason as follows:
\begin{equation*}\begin{split}
&\mathbb E_N\bigg[ \bigg(\sup_{0\leq{t}\leq{T}}\int_0^t\alpha_N 
\sum_{x\in\mathbb{T}}\eta_s(x)\eta_s(x+1){\nabla}^N h_z(x/N) ds\bigg)^2\bigg] \\
&\quad\le T \mathbb E_N\bigg[\int_0^T\alpha_N^2\Big(\sum_{x\in\mathbb{T}_N}\eta_s(x)\eta_s(x+1){\nabla}^N h_z(x/N) \big)^2ds\bigg]
\\
&\quad\leq T \alpha_N^2\int_0^T\Big(\sum_{x\in\mathbb{T}_N}\eta_s(x)^2\eta_s(x+1)^2\Big)\Big(\sum_{x\in\mathbb T_N}{\nabla}^N h_z(x/N)^2 \Big) ds
\\
&\quad\le C \alpha_N^2N^2(1+\alpha_N N) z^2
\end{split}
\end{equation*}
for some $C=C(T)>0$. Above we  used again the Cauchy-Schwarz inequality and Theorem \ref{thm:fourth-moment_bound}. 
Hence, we complete the proof of the desired estimates. 
\end{proof}
From the last inequality the condition (1) of the corollary holds. The proof is the same as for the volume so that we omit it and we only present the proof of condition (2).
\begin{corollary}
Assume $m>5/2$ and $\kappa\ge 1$. Then:
\begin{equation*}
\limsup_{N\rightarrow{+\infty}}
\mathbb{E}_{N}\Big[\sup_{0\leq{t}\leq{T}}\|\pi_{t}^{e,N}\|_{-m}^{2}\Big] < \infty 
\end{equation*}
and 
\begin{equation*}
\lim_{n\rightarrow{+\infty}}
\limsup_{N\rightarrow{+\infty}}
\mathbb{E}_{N}\Big[\sup_{0\leq{t}\leq{T}}\sum_{|z|\geq{n}}(\langle \pi_{t}^{e,N},h_{z}\rangle)^{2}\gamma_{z}^{-m}\Big]=0.
\end{equation*}
\end{corollary}
Note that condition (1) above holds as a consequence of item (a) of the previous corollary. It remains now to prove (2) but this follows from the next lemma:
\begin{lemma}
For every $n\in{\mathbb{N}}$ and every $\varepsilon>{0}$,
\begin{equation*}
\lim_{\delta\rightarrow{0}}\limsup_{N\rightarrow{+\infty}}
\mathbb{P}_N \bigg(\sup_{\substack{|s-t|<\delta, \\0\leq{s,t}\leq{T}}} \,
\sum_{|z|\leq{n}}(\langle \pi^{e,N}_{t}-\pi^{e,N}_{s},h_{z}\rangle)^{2}\gamma_{z}^{-m}>\varepsilon\bigg)=0 .
\end{equation*}
\end{lemma}
\begin{proof}
The lemma follows from showing that
\begin{equation*}
\lim_{\delta\rightarrow{0}}\limsup_{N\rightarrow{+\infty}}
\mathbb{P}_N \bigg(\sup_{\substack{|s-t|<\delta,\\0\leq{s,t}\leq{T}}}\, 
(\langle \pi^{e,N}_{t}-\pi^{e,N}_{s},h_{z}\rangle)^{2}>\varepsilon\bigg)=0.
\end{equation*}
for every $z\in{\mathbb{Z}}$ and $\varepsilon>0$.
Recalling \eqref{eq:dynkin_energy} it follows from the next
claim.
For any function $f\in{C^\infty(\mathbb{T})}$ and  for every $\varepsilon>0$ it holds 
\begin{equation*}
\lim_{\delta\rightarrow{0}}
\limsup_{N\rightarrow{+\infty}}
\mathbb{P}_N \bigg( \sup_{\substack{|s-t|<\delta,\\0\leq{s,t}\leq{T}}} \,
|M^{e,N}_t(f)-\mathcal M_{s}^{e,N}(f)|>\varepsilon \bigg) =0.
\end{equation*}
Note that 
\begin{equation*}\begin{split}
\sup_{\substack{t}}|M^{e,N}_t(f)-\mathcal M_{t_{-}}^{e,N}(f)|&=\sup_{\substack{t}}|\langle \pi_{t}^{e,N},f\rangle -\langle \pi _{t_{-}}^{e,N},f\rangle |\\&\leq \frac{\|\nabla{f}\|_{\infty}}{N^{2}}\sup_t\sum_{x\in\mathbb T_N}\eta_t(x)^2=\frac{\|\nabla{f}\|_{\infty}}{N^{2}}\sum_{x\in\mathbb T_N}\eta_0(x)^2.\end{split}
\end{equation*}
In the last line we used the conservation law. Now it is enough to apply Markov's inequality and use the assumption \eqref{eq:assumption_initial_meaure_energy} to show that this term does not contribute to the limit. 
Finally note that 
\begin{equation*}
\omega_{\delta}(M^{e,N}(f))\leq{2\omega'_{\delta}(M^{e,N}(f))+\sup_{\substack{t}}|M^{e,N}_t(f)- M_{t_{-}}^{e,N}(f)|},
\end{equation*}
and thus, the proof ends if we show that
\begin{equation*}
\lim_{\delta\rightarrow{0}}
\limsup_{N\rightarrow{+\infty}}
\mathbb{P}_N \Big(\omega'_{\delta}(M^{e,N}(f))>\varepsilon\Big)=0 
\end{equation*}
for every $\varepsilon>0$.
Again by  Aldous' criterion (\cite[Proposition 4.1.6]{kipnis1999scaling}) it is enough to show that: 
\begin{equation*}
\lim_{\delta\rightarrow{0}}
\limsup_{N\rightarrow{+\infty}}
\sup_{\substack{\tau\in{\mathfrak
{T}_{\tau}}\\0\leq{\theta}\leq{\delta}}}
\mathbb{P}_N \Big(|M_{\tau+\theta}^{e,N}(f)- M_{\tau}^{e,N}(f)|>\varepsilon\Big)=0 
\end{equation*}
for every $\varepsilon>0$. 
Here $\mathfrak {T}_{\tau}$ denotes the family of all stopping times, with respect to the canonical filtration, bounded
by $T$. From Chebychev's inequality and the Optional Sampling Theorem, we can bound
\begin{equation*}
\begin{split}
\mathbb{P}_N \Big(|M_{\tau+\theta}^{v,N}(f)- M_{\tau}^{e,N}(f)|>\varepsilon\Big)
&\le \frac{1}{\varepsilon^2} \mathbb{E}_{N}\Big[(M_{\tau+\theta}^{e,N}(f))^{2}-(M_{\tau}^{e,N}(f))^{2}\Big]
\\&\le \frac{C}{N^2\varepsilon^2}\mathbb E_N\Big[\int_\tau^{\tau+\tau'} 
\sum_{x\in\mathbb T_N}\eta_s(x)^4ds \Big]
\end{split}
\end{equation*}
where we used \eqref{eq:QV energy} in the second esimate.  
From \cref{thm:fourth-moment_bound}, after extending the interval of the integral to $[0,T]$, we conclude that the utmost right-hand side of the last display vanishes as $N\to+\infty$, which ends the proof.  
\end{proof}

\subsection{Characterization of the limit points} 
\label{subsec:clp}
We characterize now the limit points of the sequences $\{Q^{v,N}_m\}_{N\in\mathbb N}$ and $\{Q^{e,N}_m\}_{N\in\mathbb N}$.  
As a first step, we show that all the limit points of $\{Q^{v,N}_m\}_N$ and $\{Q^{e,N}_m\}_N$ are absolutely continuous with respect to the Lebesgue measure. 
The proof will be based on the key observation that if the empirical measure has a finite second moment, following the strategy of the proof given in \cite[Theorem 4.2]{suzuki1993hydrodynamic}.

\begin{lemma}
\label{lem:absolute_continuity}
Assume $\kappa\ge 1$ and assume the condition in \cref{thm:fourth-moment_bound}.   
Let $Q^v_m$ (resp. $Q^e_m$) be any limiting point of the sequence $\{Q^{v,N}_m\}_{N\in\mathbb N}$ (resp. $\{Q^{e,N}_m\}_{N\in\mathbb N}$) as $N\to\infty$. 
Then, for any $t\in [0,T]$, all limiting points are concentrated on trajectories that are absolute continuous with respect to the Lebesgue measure, i.e., 
\begin{equation*} 
Q^\sigma_m \Big( \pi^\sigma:
\pi^\sigma_t(du) = \sigma(t,u)du \text{ for some } \sigma(t,u) \Big) = 1
\end{equation*}
for each $\sigma\in \{v,e\}$. 
\end{lemma}
\begin{proof}
In what follows, let us denote by $E^{\sigma}_m$ the expectation with respect to $Q^\sigma_m$ for each $\sigma\in \{v,e\}$.  
Here we only show the assertion for the energy, since the proof works for the volume straightforwardly. 
To this end, we firstly claim that for all $0\leq t\leq T$,
\begin{equation}
\label{eq:L2bound_energy_empirical_measure_Fourier} 
E^e_m \bigg[ \int_0^t\sum_{n\in\mathbb Z} |\widehat\pi^e_s(n)|^2ds\bigg] <+\infty
\end{equation}
where $\widehat\pi^e_t(n)=\int_{\mathbb T}e^{2\pi i nu}\pi^e_t(du)$ for any $n\in\mathbb Z$.
To show the last bound, let $q_t(x)=(2\pi t)^{-1/2} e^{-x^2/(2t)}$ denote the probability density function of a centered normal distribution $\mathcal{N}(0,t)$. 
By the Cauchy-Schwarz inequality  and the boundedness of the density $q_t(x)$, we have that, for every $\delta>0$, 
\begin{equation*}
\begin{split}
\mathbb E_N\bigg[\int_0^t\frac{1}{N^2}\sum_{x,y\in\mathbb T_N}q_\delta\Big(\frac{x-y}{N}\Big)\eta_s(x)^2\eta_s(y)^2ds\bigg] 
\le\mathbb E_N\bigg[\int_0^t\frac{1}{N}\sum_{x\in\mathbb T_N}\eta_s(x)^4ds \bigg]. 
\end{split}
\end{equation*}
By Theorem \ref{thm:fourth-moment_bound}, this is bounded by $C_0+\alpha_N N$, and since we assume that $\kappa>1$, this stays finite as $N\to\infty$. 
On the other hand, by Fatou's lemma, we have that 
\begin{equation*}
\begin{split}
&\liminf_{N\to\infty}\mathbb E_N\bigg[ \int_0^t\frac{1}{N^2}\sum_{x,y\in\mathbb T_N}q_\delta\Big(\frac{x-y}{N}\Big)\eta_s(x)^2\eta_s(y)^2ds \bigg]\\
&\quad\ge E^e_m \bigg[\int_0^t\int_\mathbb T \int_\mathbb T q_\delta(u-u')\pi^e_s(du) \pi^e_s(du')ds \bigg] 
\\&\quad
= E^e_m \bigg[\int_0^t
\sum_{n\in\mathbb Z} |\widehat\pi^e_s(n)|^2 e^{-(\delta/2)(2\pi n)^2}ds \bigg],
\end{split}
\end{equation*}
along some subsequence if necessary, where we used the tightness of the empirical measure of energy in the first estimate, whereas the last identity follows by inserting the Fourier inversion of the Gaussian kernel $q_t(\theta)=\sum_{n\in\mathbb Z}\widehat q_t(n)e^{2\pi in\theta}$ where 
\begin{equation*}
\widehat q_t(n)=\int_\mathbb T q_t(\theta)e^{-2\pi i n\theta}d\theta=e^{-(t/2)(2\pi n)^2}.
\end{equation*}
Hence, by taking $\delta\to0$, by the dominated convergence theorem, the bound \eqref{eq:L2bound_energy_empirical_measure_Fourier} is in force.  

Now it is not hard to show the absolute continuity of the limiting measure, once the $L^2$-bound of the measure is proved. 
Indeed, for any $G\in C(\mathbb T:\mathbb R)$, let $\widehat G(n)=\int_{\mathbb T}e^{-2\pi i nu}G(u)du$ be the Fourier transformation of $G$ for any $n\in\mathbb Z$.
Then, by a standard Fourier analysis, we can write 
\begin{equation*}
\langle\pi^e_t, G\rangle
=\int_{\mathbb T} G(u)\pi_t^e(du)
=\int_{\mathbb T} \sum_{n\in\mathbb Z}
\widehat G(n)e^{2\pi i nu}
\pi_t(du)
=\sum_{n\in\mathbb Z} \widehat G(n) \widehat \pi^e_t(n).
\end{equation*}
Therefore, by the Cauchy-Schwarz inequality,
\begin{equation*}
E^e_m \bigg[ \int_0^t|\langle\pi^e_s,G\rangle|^2ds\bigg]
\le E^e_m\bigg[ \int_0^t\|\widehat{G}\|^2_{\ell^2(\mathbb Z)} \|\widehat \pi_s^e\|^2_{\ell^2(\mathbb Z)} ds\bigg] 
\le \|G\|^2_{L^2(\mathbb T)} 
E^e_m \bigg[ \int_0^t\|\widehat \pi_s^e\|^2_{\ell^2(\mathbb Z)} ds \bigg],
\end{equation*}
which is finite from the assertion in the beginning, where we used the Plancherel identity. 
Now, it is not hard to show that the measure $\pi^e_t$ is absolute continuous with respect to the Lebesgue measure. 
Indeed, first, note that in order to show the absolute continuity of the measure $\pi^e_t$, it is enough to show the following: for any $\varepsilon>0$, there exists some $\delta>0$ such that for any family of disjoint open intervals $\{(a_k,b_k)\}_{k=1,\ldots,m}$ we have $\pi^e_t(A)<\varepsilon$ provided $|A|< \delta$, where we set $A=\bigcup_{k=1}^m (a_k,b_k)$ and $|\cdot|$ stands for the Lebesgue measure.   
In what follows, let $F$ be the closure of the open set $A$ and let us take an approximation of $F$ by closed sets $\{F_n\}_{n\in\mathbb N} \subset \mathbb T$ at level $n\in\mathbb N$, i.e., 
\begin{equation*}
F_n = \{ x\in\mathbb T : \inf_{y\in F}|x-y| \le 1/n\}. 
\end{equation*}
Moreover, let $g_n \in C(\mathbb T: [0,1])$ be such that $g_n(x)=1$ when $x\in F$ whereas $g_n(x)=0$ for $x\in F_n^c$. 
Then, we have that 
\begin{equation*}
\begin{aligned}
E^e_m\bigg[ \int_0^t\pi^e_s(A)ds \bigg]
& \le E^e_m\bigg[ \int_0^t\pi^e_s(F)ds \bigg]
= E^e_m\bigg[ \int_0^t \langle \pi^e_s,\mathbf{1}_F\rangle ds \bigg] \\
&\le E^e_m\bigg[ \int_0^t \langle \pi^e_s,g_n\rangle ds \bigg]
\le C\|g_n\|_{L^2(\mathbb T)}
\le C|F_n|
\end{aligned}
\end{equation*}
with some $C>0$. 
Noting $F_n \searrow F$, we have that the utmost left-hand side of the last display is in fact bounded from above by $C|A|$. 
Therefore, we conclude the absolute continuity of $\pi^e_t$ at any time $t$. 
\end{proof}

Now, as a consequence of the absolute continuity of the limiting measures, it is easy to see that any limit point is concentrated on trajectories satisfying the weak form of the hydrodynamic equation. 
More precisely, we notice firstly for the volume that every limit point $Q^v_m$ of the sequence $\{Q^{v,N}_m\}_N$ is concentrated on trajectories $\pi^v=\{\pi^v_t: t\in[0,T]\}$ satisfying the weak form of the first equation of \eqref{eq:hdl}, that is, for any test function $G\in C^2(\mathbb T)$, we have that
\begin{equation*}
Q^v_m\bigg( \pi^v \, :\, 
\langle \pi^v_t,G\rangle
=\langle\pi^v_0,G\rangle
+\int_0^t\langle\pi^v_s,\Delta G\rangle ds -2\alpha\int_0^t\langle \pi^v_s,\nabla G\rangle ds \bigg) = 1 
\end{equation*}
when $\kappa=1$, whereas the same result holds with $\kappa>1$ with $\alpha=0$.
Analogously, for the energy, assume $\kappa>1$.
Then, every limit point $Q^e_m$ of the sequence $\{Q^{e,N}_m\}_N $ is concentrated on trajectories $\pi^e=\{\pi^e_t: t\in[0,T] \}$ satisfying the weak form of the second equation of \eqref{eq:hdl}: 
\begin{equation*}
Q^e_m \bigg( \pi^e \, :\,
\langle \pi^e_t,G\rangle
=\langle\pi^e_0,G\rangle
+\int_0^t\langle\pi^e_s,\Delta G\rangle ds
\bigg) = 1. 
\end{equation*}
Finally, let us comment here that the convergence along the full sequence follows from the uniqueness of the weak solution limiting equation.

\subsection{Quantitative hydrodynamic limit for energy}
\label{subsec:HDL_quantitative_energy}
In this subsection, we give a proof of \cref{thm:HDL_energy}.  
Our proof relies on the analysis of the volume-volume correlation function and the estimates obtained in Section \ref{sec:two-point_correlation_estimate}. 
For $t\in[0,T]$ and $x\in\mathbb T_N$, let $\gamma_t^N(x)$ be defined by 
\begin{equation*}
\gamma_t^N(x)
= e^N_t(x)-e(t,x/N) 
\end{equation*}
where the second term is defined via the solution of \eqref{eq:hdl} with initial profile $e_0$. 
Then, we have the following estimate for the discrepancy between the discrete and continuous energy profiles, from which we conclude the proof of \cref{thm:HDL_energy}. 

\begin{lemma}
There exists some $C=C(T,v,e)>0$ such that  
\begin{equation}
\label{eq:limit_energy profike}
\sup_{0\le t\le T}\max_{x\in\mathbb T_N}
| \gamma^N_t(x)| \le C\frac{\log N}{N} . 
\end{equation}
\end{lemma}
\begin{proof} 
It is easy to check that the energy element $e_t^N$ satisfies  
\begin{equation}
\label{eq:disc_ene}
\partial_t e^N_t(x)
= \Delta^N e^N_t(x)
+ 2\alpha_NN 
\nabla^N \mathbb E_N\big[\eta_t(x-1)\eta_t(x)\big].
\end{equation}
Then, a simple computation shows that
\begin{equation*}
\begin{split}
\partial_t\gamma_t^N(x)
&=\Delta^N\gamma^N_t(x) 
+\Delta^Ne(t,x/N) -\Delta e(t,x/N)
+2\alpha_N N^2 \big(\varphi_t^N(x,x+1)-\varphi_t^N(x-1,x)\big)\\
&\quad+ 2\alpha_NN \nabla^N\big( v^N_t(x-1) v_t^N(x)\big)-2\alpha\nabla v(t,x/N)^2.
\end{split}
\end{equation*}
Note that since the energy profile $e$ is smooth in the spatial variable, by Taylor's theorem, we have that 
\begin{equation*}
\sup_{0\le t\le T} \max_{x\in\mathbb T_N} 
\big| 
\Delta^Ne(t,x/N) -\Delta e(t,x/N)\big| 
\le C
N^{-2}
\end{equation*}
for some constant $C=C(T,e)>0$. 
Moreover, since $v^N_t(x)$ and $v(t,u)$ solve, respectively, the discrete and continuous heat equations, it is not hard to check that 
\begin{equation}
\label{eq:volume_discrepancy_disc_conti}
\sup_{0\leq t\leq T}
\max_{x\in\mathbb T_N} 
| v_t^N(x)- v(t,x/N) |
\le CN^{-2} 
\end{equation}
with some constant $C=C(T,v)>0$.  
Moreover, note that we have the decomposition 
\begin{equation*}
\begin{aligned}
& \nabla^N\big( v^N_t(x-1) v_t^N(x)\big)-\nabla v(t,x/N)^2\\
&\quad= \big[ \nabla^N\big( v^N_t(x-1) v_t^N(x)\big)
- \nabla^Nv_t^N(x)^2\big] 
+ \big[ \nabla^Nv_t^N(x)^2
- \nabla^N v(t,x/N)^2\big] \\
&\qquad+ \big[ \nabla^N v(t,x/N)^2
- \nabla v(t,x/N)^2\big] . 
\end{aligned}
\end{equation*}
Then, we can show that the error bound \eqref{eq:volume_discrepancy_disc_conti} propagates to the second and third terms on the right-hand side of the last display, whereas for the first term is estimated by $O(N^{-1})$ by a simple computation. 
Hence, we have the bound 
\begin{equation*}
\sup_{0\le t\le T} \max_{x\in\mathbb T_N} 
\big| 
\nabla^N\big( v^N_t(x-1) v_t^N(x)\big)-\nabla v(t,x/N)^2\big| 
\le C N^{-1}
\end{equation*}
with some $C=C(T,v)>0$. 
By this line, we have that 
\begin{equation*}
\partial_t\gamma_t^N(x)
=\Delta^N\gamma^N_t(x)+2\alpha_N N^2 \big( \varphi_t^N(x+1,x)-\varphi_t^N(x,x-1)\big)
+ R^N_t(x),
\end{equation*}
where the reminder term $R^N_t(x)$ satisfies the bound 
\begin{equation*}
\sup_{0\leq t\leq T}
\max_{x\in\mathbb T_N} 
| R^N_t(x)| \le CN^{-1}
\end{equation*}
with some $C=C(T,v,e)>0$. 
To proceed further, let $\{Y^{(1)}_t:t\ge 0\}$ be a continuous-time symmetric one-dimensional random walk on $\mathbb T_N$, generated by $\Delta^{(1)}_N$, that is, the one with rate $N^2$ in all directions, starting from $x\in\mathbb T_N$, and additionally, we denote by $P_x$ the probability associated to this random walk and write the expectation with respect to $ P_x$ by $E_x$. 
Then, by Duhamel's principle, we have the following random walk representation: 
\begin{equation*}
\begin{aligned}
\gamma_t^N(x)
&= E_x\big[\gamma_0^N(Y^{(1)}_t)\big] \\
&\quad+2\alpha_N N^2\int_0^t\sum_{y\in\mathbb T_N}
\big( \varphi^N_{t-s}(y+1,y)- \varphi^N_{t-s}(y-1,y) \big)
P_{x}^N(Y^{(1)}_s=y)ds 
+R^N_t(x). 
\end{aligned}
\end{equation*}
By summation by parts applied to the last display, we have that  
\begin{equation*}
\begin{aligned}
\sup_{0\leq t\leq T}
\|\gamma^N_t\|_{\ell^\infty(\mathbb T_N)}  
&\le\|\gamma_0^N\|_{\ell^\infty(\mathbb T_N)} 
+2\alpha_N N\bigg|\int_0^t\sum_{y\in\mathbb T_N}\varphi^N_s(y+1,y)
\nabla^N_y P_{x}^N(Y^{(1)}_s=y)ds\bigg| \\
&\quad+ \sup_{0\leq t\leq T}
\|R^N_t\|_{\ell^\infty(\mathbb T_N)} .
\end{aligned}
\end{equation*}
Above, we used the notation $\nabla^N_y$ to emphasize that the discrete derivative $\nabla^N$ is acting on the variable $y$. 
Now, in order to estimate the above integral, we will consider a fixed small time $t_0>0$ that will be optimized later, splitting the integral in the short-time interval $[0,t_0)$ and the rest $[t_0,t)$. 
For the short-time interval we have that
\begin{equation*}
\bigg| \int_0^{t_0}\sum_{y\in\mathbb T_N}\varphi^N_s(y+1,y) \nabla^N_y P_x(Y^{(1)}_s=y)ds \bigg|
\le 2t_0,
\end{equation*}
where we used \cref{lem:correlation_estimate} and the fact that $P_x$ is a probability satisfying $\sum_{y} P_x(Y^{(1)}_t=y)=1$ for any $t$ and $x$. 
On the other hand, for long times, by \cref{lem:SRW_transition_prob_estimate} and \cref{lem:correlation_estimate}, we have the bound 
\begin{equation*}
\begin{aligned}
\bigg|\int_{t_0}^{t}\sum_{y\in\mathbb T_N}\varphi^N_s(y+1,y) \nabla^N_y P_x(Y_s^{(1)}=y)ds\bigg|
\le \int_{t_0}^{t}\frac{C}{N s}ds,
\end{aligned}
\end{equation*}
which is of order $O(N^{-1}\log (1/t_0))$.  
Thus, 
\begin{equation*}
\sup_{0\leq t\leq T}\|\gamma^N_t\|_{\ell^\infty(\mathbb T_N)}
\le\|\gamma_0^N\|_{\ell^\infty(\mathbb T_N)} 
+ C\alpha_N \big(N t_0+ \log(1/t_0) \big)
+ CN^{-1} .
\end{equation*}
From the last bound, by taking $t_0=N^{-\delta}$ with some $\delta>1$, we conclude
\begin{equation*}
\sup_{0\leq t\leq T} \|\gamma^N_t\|_{\ell^\infty(\mathbb T_N)}
\leq\|\gamma_0^N\|_{\ell^\infty(\mathbb T_N)}
+
C\frac{\log N}{N},
\end{equation*}
and thus we complete the proof.
\end{proof}

\section{Non-Equilibrium fluctuations}
\label{sec:fluctuations}
Here we give the proof of \cref{thm:main fluctuations}.
Our starting point is to obtain a martingale decomposition for the volume fluctuation field $\mathcal{V}^N_\cdot$.

\subsection{Auxiliary martingales}
By Dynkin's martingale formula, for any $G\in C^2(\mathbb T)$, the processes
\begin{equation}
\label{eq:dynkin_martingale_definition}
\mathcal{M}^{N}_t(G)
= \mathcal{V}^N_t(G)
- \mathcal{V}^N_t(G)
+ \int_0^t (\partial_s + L_N) \mathcal{V}^N_s(G)ds 
\end{equation}
and $\mathcal{M}^N_t(G)^2 - \langle \mathcal{M}^N(G)\rangle_t$ are martingales where 
\begin{equation}
\label{eq:qv_volume_fluctuation}
\begin{aligned}
\langle \mathcal{M}^N(G)\rangle_t 
&=\int_0^t(L_N \mathcal{V}^N_s(\varphi)^2 - 2 \mathcal{V}^N_s(\varphi)L_N \mathcal{V}^N_s(\varphi) \big) ds \\
&= \int_0^t \frac{1}{N}\sum_{x\in\mathbb T_N} 
\big(\eta_s(x+1)-\eta_s(x))\big)^2 \big(T^-_{f_Ns}\nabla^NG(x/N)\big)^2 ds .
\end{aligned}
\end{equation}
Now we compute the action of the generator. 
First, for the symmetric part, a simple computation gives 
\begin{equation*}
\begin{aligned}
S_N \mathcal{V}^N_t(G)
= \mathcal{V}^N_t(\Delta^NG)
+ \frac{1}{\sqrt{N}}\sum_{x\in\mathbb Z}
\big(\Delta^N  v^N_t(x)\big) T^-_{f_Nt}G(x/N). 
\end{aligned}
\end{equation*}
On the other hand, for the anti-symmetric part, we compute 
\begin{equation}
\label{eq:anti-symmetric_part_computation}
\begin{aligned}
(\partial_t + A_N) \mathcal{V}^N_t(G)
&= 2\alpha_N N \mathcal{V}^N_t(\widetilde{\nabla}^N G) 
- \frac{f_N}{N} \mathcal{V}^N_t (\nabla G) 
\\ &\quad 
+ 2\alpha_N N^{1/2}\sum_{x\in\mathbb Z}
v^N_t(x) \widetilde{\nabla}^N T^-_{f_Nt}G(x/N).
\end{aligned}
\end{equation}
Note that, recalling the definition of the moving frame \eqref{eq:moving_frame_definition}, by the Cauchy-Schwarz inequality, the first two terms on the right-hand side of \eqref{eq:anti-symmetric_part_computation} can be treated as follows
\begin{equation*}
\begin{aligned}
&\mathbb E_N \bigg[\sup_{0\le t\le T} \bigg|\int_0^t 
\Big( 2\alpha_NN \mathcal{V}^N_s(\widetilde \nabla^NG) 
- \frac{f_N}{N} \mathcal{V}^N_s(\nabla G)\Big) ds \bigg|^2 \bigg] \\
&\quad\le CT \alpha_N^2N^2 \int_0^T \mathbb E_N \Big[ 
\mathcal{V}^N_t(\widetilde \nabla^NG-\nabla G)^2\Big] dt  \\
&\quad\le CT^2\alpha_N^2N
\Big(\max_{x\in\mathbb T_N}\big|\widetilde \nabla^NG(x/N) - \nabla G(x/N)\big|^2\Big) 
\int_0^T \sum_{x,y\in\mathbb T_N} \mathbb E_N\big[\overline\eta_t(x)\overline\eta_t(y)\big] dt
\end{aligned}
\end{equation*}
for some $C>0$. 
Note that the discrepancy between the discrete derivative and the continuous one is estimated as of order $O(N^{-2})$.
Thus, using the correlation estimate~\eqref{eq:volume_correlation_estimate}, we can easily see that the utmost right-hand side of the last display vanishes as $N\to \infty$.  
Hence, we have the following decomposition for the volume fluctuation field:
\begin{equation}
\label{eq:martingale_decomposition}
\mathcal{V}^N_t(G)
= \mathcal{V}^N_0(G)
+ \int_0^t \mathcal{V}^N_s(\Delta^NG) ds 
+ \mathcal{M}^N_t(G) 
+ \mathcal R^N_t(G)
\end{equation}
where the reminder term $\mathcal R^N_\cdot$ is negligible in the following sense: 
\begin{equation}\label{eq:remainder}
\limsup_{N\to+\infty}
\mathbb E_N\Big[ \sup_{0\le t \le T} 
|\mathcal R^N_t(G)|^2\Big] = 0
\end{equation}
for any $G\in C^\infty(\mathbb T)$.

\subsection{Tightness}\label{sec:flu:tight}
In what follows, we show that each term in the decomposition \eqref{eq:martingale_decomposition} is tight.

\begin{lemma}
\label{lem:tightness_fluctuation}
Assume $m>5/2$.
Then, the sequence of volume fluctuation fields $\{\mathcal V^N_t\}_{N\in\mathbb N}$ is tight in $D([0,T],H^{-m}(\mathbb T))$. 
\end{lemma}

The strategy is analogous to \cref{subsec:HDL_tightness}, so that we only present the crucial estimates here.  
First, note that we have the following key estimate to show the tightness.
Here, recall from \eqref{eq:sobolev_CONS_definition} the definition of the function $h_z$. 

\begin{lemma}
There exists some $C=C(T)>0$ such that for any $z\in\mathbb Z$,
\begin{equation*}
\limsup_{N\to+\infty}
\mathbb E_N\Big[ \sup_{0\le t\le T} 
|\mathcal V^N_t(h_z)|^2 \Big]
\le C(1+z^4).
\end{equation*}
\end{lemma}
\begin{proof}
Analogously to the proof of \cref{lem:HDL_tighness_key_estimate}, let us estimate each term in the martingale decomposition \eqref{eq:martingale_decomposition}.  
First, for the initial fields, we can easily see that 
\begin{equation*}
\limsup_{N\to+\infty}
\mathbb E_N\big[|\mathcal V^N_0(h_z)|^2 \big]
\le C  
\end{equation*}
for some $C>0$ with the help of the correlation estimate (\cref{lem:correlation_estimate}) and the assumption \eqref{eq:assumption_initial_meaure_energy}.  
Next, for the integral term, we have, by the Cauchy-Schwarz inequality, that 
\begin{equation*}
\begin{aligned}
&\mathbb E_N\bigg[ \sup_{0\le t \le T} 
\bigg| \int_0^t \frac{1}{\sqrt{N}} 
\sum_{x\in\mathbb T_N} \overline{\eta}_s(x)
T^-_{f_Ns}\Delta^N h_z(x/N)ds\bigg|^2 \bigg]\\
&\le T \int_0^T \frac{1}{N} \sum_{x,y\in\mathbb T_N}
\mathbb E_N[\overline\eta_t(x)\overline\eta_t(y)] 
\Big( T^-_{f_Nt}\Delta^Nh_z(x/N)\Big)
\Big(T^-_{f_Nt}\Delta^N h_z(y/N) \Big) dt . 
\end{aligned}
\end{equation*}
With the help of the correlation estimate \cref{lem:correlation_estimate}, we can easily see that the right-hand side of the last display is bounded by $CT^2z^4$ with some universal constant $C>0$.  

Finally, for the martingale term, we have that 
\begin{equation*}
\begin{aligned}
\mathbb E_N\Big[\sup_{0\le t\le T} 
| \mathcal M^N_t(h_z)|^2 \Big]
\le 4\mathbb E_N\big[| \mathcal M^N_T(h_z)|^2 \big]
= 4\mathbb E_N\big[\langle \mathcal M^N(h_z)\rangle_T \big]
\end{aligned}
\end{equation*}
where we used Doob's inequality in the first estimate. 
Now recalling from \eqref{eq:qv_volume_fluctuation} the computation of the quadratic variation we have that 
\begin{equation*}
\begin{aligned}
\mathbb E_N\big[\langle \mathcal M^N(h_z)\rangle_T \big]
= \int_0^T \frac{1}{N}\sum_{x\in\mathbb T_N} 
\mathbb E_N\Big[\big(\eta_t(x+1)-\eta_t(x))\big)^2\Big] 
\big(T^-_{f_Nt}\nabla^N h_z(x/N)\big)^2 dt.
\end{aligned}
\end{equation*}
Since the right-hand side of the last display is bounded by $CTz^2$ with some $C>0$, we end the proof. 
\end{proof}

Following the proof in \cref{subsec:HDL_tightness}, we know that now it is enough to show the tightness of the sequence of martingales $\{ \mathcal{M}^N_t(G):t\in [0,T]\}_{N\in\mathbb N}$ with the help of the Aldous criterion (\cite[Proposition 4.1.6]{kipnis1999scaling}).   
To that end, recall that the quadratic variation process is computed as \eqref{eq:qv_volume_fluctuation}. 
Now, for any stopping time $\tau\in\mathfrak{T}_T$, it holds that 
\begin{equation*}
\begin{aligned}
&\mathbb P_N \big( 
\big| \mathcal{M}^N_{\tau+\gamma}(G)-\mathcal{M}^N_\tau(G)\big| > \varepsilon \big)\\
&\quad\le \varepsilon^{-2} 
\mathbb E_N \Big[ | \mathcal{M}^N_{\tau+\gamma}(G)-\mathcal{M}^N_\tau(G)|^2 \Big]\\
&\quad\le \varepsilon^{-2} 
\mathbb E_N\bigg[
\int_\tau^{\tau+\gamma} 
\frac{1}{N}\sum_{x\in\mathbb T_N} 
\big(\eta_s(x+1)-\eta_s(x))\big)^2 \big(T^-_{f_Ns}\nabla^N G(x/N)\big)^2 
ds \bigg] .
\end{aligned}
\end{equation*}
According to the assumption \eqref{eq:assumption_initial_meaure_energy} and the last estimate, we conclude that 
\begin{equation*}
\lim_{\delta\to0} \limsup_{N\to\infty}
\sup_{\gamma \le \delta}
\sup_{\tau\in\mathcal{T}_T}
\mathbb P_N\big( | \mathcal{M}^N_{\tau + \gamma}(G) - \mathcal{M}^N_\tau(G) | > \varepsilon \big)
= 0
\end{equation*}
for any $\varepsilon > 0$, and thus the second condition of Aldous' criterion follows.
Furthermore, note that 
\begin{equation*}
\mathbb E_N\big[ \mathcal{M}^N_t(G)^2 \big]
=\mathbb E_N\big[ \langle\mathcal{M}^N(G)\rangle_t \big]
\le  2C_E \| \nabla G\|^2_{L^\infty(\mathbb T)} t,
\end{equation*}
so that the sequence $\{ \mathcal{M}^N_t(G)\}_N$ is uniformly bounded in $L^2(\mathbb P_N)$ for each $t\in [0,T]$ and thus the first condition of the Aldous' criterion is straightforward. 
Hence, the tightness of the martingale sequence is verified. 

Now that we have all the ingredients for the proof of the theorem we analyze now each bullet point. 

\subsection{Proof of item a)}

From the previous subsection we know that the sequence of mean-zero martingales $\mathcal M^N_t$ is tight with respect to the uniform topology of $D([0,T], H^{-m})$ for $m>5/2$. Moreover, this sequence of martingales is uniformly integrable as their second moments  are bounded. Therefore, there exists a sequence along which it converges to a limit $\mathcal M_t$ which again is a mean-zero martingale.

\subsection{Proof of item b)}
In this section we complement our knowledge about the limit of the martingale sequence of item a). We  compute now the limit in mean of the sequence of martingales $\{\langle \mathcal{M}^N_t(G)\rangle\}_{N\in\mathbb N}$. 
Recall \eqref{eq:qv_volume_fluctuation}. Then,   
\begin{equation*}
\begin{aligned}
&\mathbb E_N\big[\langle \mathcal{M}^N(G)\rangle_t \big] \\
&\quad= \int_0^t  
\frac{1}{N}\sum_{x\in\mathbb T_N} 
\mathbb E_N\Big[ \big(\eta_s(x+1)-\eta_s(x))\big)^2 \Big] 
\big(T^-_{f_Ns}\nabla^N G(x/N)\big)^2 ds\\
&\quad= \int_0^t \frac{2}{N}\sum_{x\in\mathbb T_N} 
\big(e^N_s(x)+e_s^N(x+1)-v^N_s(x)^2-v^N_s(x+1)^2\big)
\big(T^-_{f_Ns}\nabla^N G(x/N)\big)^2 ds\\
&\qquad- \int_0^t \frac{2}{N}\sum_{x\in\mathbb T_N} 
\varphi^N_s(x,x+1)
\big(T^-_{f_Ns}\nabla^N G(x/N)\big)^2 ds.
\end{aligned}
\end{equation*} 
Now, using \eqref{eq:convergece_profile_energy} together with \eqref{eq:volume_correlation_estimate} and the fact that the volume discrete profile $\{v_t^N(x): t\ge 0, x\in\mathbb T_N\}$ converges to the solution of the heat equation $\{v(t,u): t\ge 0 ,u\in\mathbb T\}$, the last display converges as $N\to+\infty$ to \eqref{eq:lim_qv}. 

\subsection{Proof of item c)}
Fix a time $t\in[0,T]$  and recall \eqref{eq:martingale_decomposition}. Take there the test function $G_s\coloneqq \mathcal T_{t-s}G$ where $T_{t-s}G$ is the solution of the heat equation with initial condition $\mathcal T_tG$. In this case, since the test function is dependent on the time variable, we rewrite \eqref{eq:martingale_decomposition} as 
\begin{equation}
\label{eq:martingale_decomposition_new}
\mathcal{V}^N_t(G)
= \mathcal{V}^N_0(T_tG)
+ \int_0^t \mathcal{V}^N_s(\Delta^N \mathcal T_{t-s}G+\partial_s \mathcal T_{t-s}G) ds 
+ \mathcal{M}^N_t(G) 
+ \mathcal R^N_t(G). 
\end{equation}
By the definition of $\mathcal T_{t-s} G$ the integral term in the last identity is null. We also note that   \eqref{eq:remainder} also holds in this case. 
Moreover, in Subsection \ref{sec:flu:tight} we have already seen that the sequence of measures $\{\mathcal Q^N_m\}_N$ is tight in $D([0,T]: H^{-m})$ for $m>5/2$.
Consequently, taking the limit through a subsequence  in \eqref{eq:martingale_decomposition_new}, the limit point satisfies:
\begin{equation*}
\mathcal{V}_t(G)
= \mathcal{V}_0(G)
+ \mathcal{M}_t(G) . 
\end{equation*}
Above we put the dependence on the limit martingale on the gradient of $G$ as a consequence of the computations on the previous subsection. 
Let us show now that any limit point $\mathcal M_\cdot$ of the sequence $\{\mathcal{M}^N_\cdot\}_N$ has continuous trajectories with probability one. We will denote by $\Delta^N_t$ the size of the largest jump in $H^{-m}(\mathbb T_N)$, which is given by 
\begin{equation*}
    \Delta_T^N\coloneqq\sup_{0\leq t\leq T}\|\mathcal M^N_t-\mathcal M^N_{t-}\|_{-m}.
\end{equation*} 
By the martingale decomposition \eqref{eq:martingale_decomposition}, we have that the jumps of $\mathcal M^N_t$ are the ones of the field $\mathcal V^N_t$. Thus, 
\begin{equation*}
    \Delta^N_T=\sup_{0\leq t\leq T}\sum_{z\in\mathbb Z}\gamma_z^{-m}
    \big( \mathcal V^N_t(h_z)-\mathcal V^N_{t-}(h_z) \big)^2.
\end{equation*}
Let $x_t\in\mathbb T_N$ denote the point at which there was an update of the exchange dynamics at the time $t\in [0,T]$. By the definition of the volume fluctuation field, we get that 
\begin{equation*}
    \mathcal V^N_t(h_z)-\mathcal V^N_{t-}(h_z)=\frac{1}{N^{3/2}}\big[\eta_t(x_t+1)-\eta_t(x_t)+v^N_t(x_t)-v^N_t(x_t+1)\big]\nabla^Nh_z(x_t/N), 
\end{equation*}
which implies that
\begin{equation}
\label{eq:cont.martingale}
\begin{split}
     \mathbb E_N\big[\Delta_T^N\big]
     &\le \frac{2}{N^{3}}\mathbb E_N\bigg[\sup_{0\leq t\le T}\sum_{z\in\mathbb Z}\gamma_z^{-m}\big( \eta_t(x_t+1)-\eta_t(x_t)\big)^2\big(\nabla^Nh_z(x_t/N)\big)^2\bigg]\\
     &\quad+\frac{2}{N^{3}} \mathbb E_N\bigg[ \sup_{0\leq t\leq T}\sum_{z\in\mathbb Z}\gamma_z^{-m} \big( v^N_t(x_t)-v^N_t(x_t+1)\big)^2\big(\nabla^Nh_z(x_t/N)\big)^2\bigg] . 
\end{split}
\end{equation}
For the first term, note that by the conservation of the energy
\begin{equation*}
    (\eta_t(x_t+1)-\eta_t(x_t))^2\leq\sum_{x\in\mathbb T_N}(\eta_t(x+1)-\eta_t(x))^2 \leq2\sum_{x\in\mathbb T_N}\eta_t(x)^2
    =2\sum_{x\in\mathbb T_N}\eta_0(x)^2. 
\end{equation*}
Moreover, note that the derivative of the basis $h_z$ gives a contribution of order $z$, so that the first line in \eqref{eq:cont.martingale} is bounded by 
\begin{equation*}
    \frac{C}{N^2} \|e^N_0\|^2_{\ell^\infty(\mathbb T_N)}    \sum_{z\in\mathbb Z}z^2 \gamma_z^{-m}
\end{equation*} 
for some universal constant $C>0$. 
Noting that the sum over $z\in\mathbb Z$ remains finite as long as $m>3/2$, the last display vanishes as $N\to+\infty$. 
Additionally, using the fact that $v_t^N$ is the solution of the discrete heat equation, we can easily check also that the second line in \eqref{eq:cont.martingale} vanishes, provided $m>3/2$. 
Therefore, we conclude that for any $\varepsilon>0$, 
$$
\limsup_{N\to\infty}\mathbb P_N(\Delta_T^N>\varepsilon)=0.
$$
This immediately shows that the all limit points are concentrated on continuous trajectories. 

Finally, to prove that $\mathcal V_0$ and $\mathcal M_t$ are uncorrelated, fix  $f,g\in\mathcal D(\mathbb T)$  and note that by the martingale property  we conclude that  for any $t\in[0,T]$, 
\begin{equation*}
E_{\mathcal Q_m} \big[ \mathcal  V_0(g)\mathcal M_t(f)\big]
= E_{\mathcal Q_m}\big[ E_{\mathcal Q_m}\big[ \mathcal  V_0(g)\mathcal M_t(f)\big| \mathcal F_0\big] \big]
=E_{\mathcal Q_m}\big[ \mathcal V_0(g) E_{\mathcal Q_m}\big[ \mathcal M_t(f)|\mathcal F_0\big]\big] =0
\end{equation*}
where $\mathcal F_0$ denotes the $\sigma$-fields generated by $\{\eta_0(x)\}_{x\in\mathbb T_N}$ and $E_{\mathcal Q_m}$ denotes the expectation with respect to the measure $\mathcal Q_m$. 

\subsection{Proof of item d)}
Assume now that $\mu_N$ is the invariant measure $\mu_{\rho,\beta}$ for some $\rho\in\mathbb R$ and $\beta>0$. In this case, we can really characterize the law of the martingale $\mathcal M_t$ since we can prove the next result.
Indeed, first note that the field $\mathcal V_t^N$ converges, through a subsequence, as it is tight, for any time $t$ to a space white noise with variance $1/\beta$. This is a consequence of the fact that the dynamics is starting from the product, spatially homogeneous invariant measure $\mu_{\rho,\beta}$ and the proof goes through the Fourier transform. We leave the details to the reader as they are classical. 
Moreover, note that 
\begin{equation*}
\begin{aligned}
&\mathbb E_N\Big[\Big(\langle \mathcal{M}^N(G)\rangle_t -\mathbb E_N[\langle \mathcal{M}^N(G)\rangle_t]\Big)^2\Big]\\ 
&= \mathbb E_N\Big[\Big(\int_0^t  
\frac{1}{N}\sum_{x\in\mathbb T_N} 
\Big[ \big(\eta_s(x+1)-\eta_s(x)\big)^2-\frac{2}{\beta} \Big] 
\big(T^-_{f_Ns}\nabla^N G(x/N)\big)^2 ds\Big)^2\Big]
\end{aligned}
\end{equation*} 
To conclude now it is enough to use the Cauchy-Schwarz inequality and since the variables are centered and the measure is product, the last display can be bounded from above by
\begin{equation*}
\begin{aligned}
&\mathbb E_N\Big[\Big(\langle \mathcal{M}^N(G)\rangle_t -\mathbb E_N[\langle \mathcal{M}^N(G)\rangle_t]\Big)^2\Big]\\ 
&= t\int_0^t  
\frac{1}{N^2}\sum_{x\in\mathbb T_N} 
\mathbb E_N \Big[\Big(\big(\eta_s(x+1)-\eta_s(x)\big)^2-\frac{2}{\beta} \Big)^2\Big] 
\big(T^-_{f_Ns}\nabla^N G(x/N)\big)^4 ds.
\end{aligned}
\end{equation*} 
From Assumption \ref{assump:initial_measure}
last display vanishes as $N\to+\infty.$ Now, from the previous items and a criteria listed in \cite[Proposition 4.4]{gonccalves2024clt} for instance, we conclude that the quadratic variation of the limit martingale $\mathcal M_t$ is equal to 
\begin{equation*}
\langle \mathcal M(f)\rangle_t= \int_0^t \int_{\mathbb T} \frac{2} {\beta} (\nabla f(u))^2 du ds
= 2t\beta^{-1}\| \nabla f\|^2_{L^2(\mathbb T)}, 
\end{equation*}
and thus the proof follows.

\section{Proof of the fourth-moment bound}
\label{sec:moment bound}
In this section, we present a uniform bound on the fourth moment to show \cref{thm:fourth-moment_bound}. 
To that end, we recall some basic notion of $\mathcal H^{-1}=\mathcal H^{-1,N}$-norm from \cite[Section 5.6]{kipnis1999scaling}. 
First, recall here that the $\mathcal H^{-1,N}$-norm $\| \cdot \|_{-1,N}$ is given by
\begin{equation}
\label{eq:h-1norm_expression}
\|f\|^2_{-1,N}
= \frac{1}{N}\sum_{x,y\in\mathbb T_N}
f(x)K_N(x-y)f(y)^*
\end{equation}
where 
\begin{equation*}
K_N(x)
= \frac{1}{N}\sum_{z\in\mathbb T_N} \frac{\psi_z(x)}{a_N(z)}
\end{equation*}
with 
\begin{equation*}
a_N(z) 
= 1 + 2N^2\big(1-\cos(2\pi z/N)\big) 
\end{equation*}
and $\psi_z(x)=\exp(2\pi i z x/N)$ which consists of the orthonormal basis of $\ell^2(\mathbb T_N)$. 
Note that the kernel $K_N$ is an even real function using the fact that $z\mapsto 1/(1 + 2N^2 (1-\cos(2\pi z/N) )$ is even real and the definition of $\psi_z(x)$.   
Observe that following identity  holds
\begin{equation}
\label{eq:H-1_kernel_property}
(I - \Delta^N)K_N(x) = \delta_{0,x}
\end{equation}
where recall that $\Delta^N=\Delta_N^{(1)}$ is the one-dimensional discrete Laplacian and $\delta_{0,x}$ stands for the Kronecker delta. 
In what follows, we will use the following properties several times.

\begin{lemma}
We have the following properties: 
\begin{itemize}
\item[a)]
For any two real functions $f,g\in \ell^2(\mathbb{T}_N)$, and for all $A>0$, we have that 
\begin{equation}
\label{eq:H-1_kernel_young}
|\langle f,K_N*g\rangle_{\ell^2(\mathbb T_N)}|\leq \frac{A}{2}\|f\|^2_{-1,N}+\frac{1}{2A}\|g\|^2_{-1,N}.
\end{equation}
\item [b)] We have that 
\begin{equation}
\label{eq:H-1_kernel_estimate}
N^2 \big(K_N(0) - K_N(1) \big) \le 1/2. 
\end{equation}

\item[c)]
For any $x>0$, we have that 
\begin{equation}
\label{eq:H-1_kernel_identity_difference}
K_N(x)-K_N(x+1)=\frac{1}{N^2}\Big\{\frac{1-K_N(0)}{2}-\sum_{j=1}^{x}K_N(j)\Big\}.
\end{equation}
Note that the case $x<0$ can be covered using the fact that the kernel $K_N$ is even real.
\end{itemize}
\end{lemma}
\begin{proof}
For the proof of the first item, we are refereed to \cite[Section 5.6]{kipnis1999scaling}. 
To prove the second item, it is enough to note that from \eqref{eq:H-1_kernel_property} applied to $x=0$, and by using the fact that the kernel satisfies $K_N(1)=K_N(-1)$ and that $K_N(0)$ is non-negative, then we get the identity 
$$1=K_N(0)-\Delta^N K_N(0),$$ 
which is equivalent to 
$$
2N^2(K_N(0)-K_N(1))=1-K_N(0) $$ 
and this gives the inequality.  
Now we are in a position to prove the last identity c).
First, note that
\begin{equation*}
\begin{aligned}
N^2( K_N(x)-K_N(x+1))
&=N\sum_z \frac {1}{a_N(z)}\psi_z(x)(1-\exp(2\pi i z/N))\\
&=\frac{K_N(0)-1}{2}
+\sum_z \frac{N}{a_N(z)}(1-\psi_{z}(x+1))(1-\exp(-2\pi i z/N)). 
\end{aligned}
\end{equation*}
where we summed and subtracted a term $N^2(K_N(0)-K_N(-1))$, and used the item b) in the last identity. 
Now,  we use the fact that 
\begin{equation*}\begin{split}
(1-\psi_{z}(x+1))=(1-\psi_z(1))\sum_{j=0}^x\psi_z(j), 
\end{split}
\end{equation*}
which yields
\begin{equation*}\begin{split}
N^2(K_N(x)-K_N(x+1))&=\frac{K_N(0)-1}{2}-\sum_{j=0}^x\Delta^NK_N(j).
\end{split}
\end{equation*}
Now if we use the fact that $-\Delta^NK_N(0)=1-K_N(0)$ together with \eqref{eq:H-1_kernel_property} and we are done. 
\end{proof}

Now we head up to the proof of \cref{thm:fourth-moment_bound}.
By the explicit formula for the $\mathcal H^{-1}$-norm given in \eqref{eq:h-1norm_expression} we have that 
\begin{align*}
L_N \|\eta^2\|^2_{-1,N}=\frac{1}{N}\sum_{x,y\in\mathbb T_N}K_N(x-y) L_N\big(\eta(x)^2\eta(y)^2\big) .
\end{align*}
We begin by computing the contribution of the symmetric part:
\begin{align*}
& \frac{1}{N}\sum_{x,y\in\mathbb T_N} K_N(x-y)N^2
S\big( \eta(x)^2\eta(y)^2\big) \\
&\quad=\frac{1}{N}\sum_{|x-y|>1} \big\{
\eta(y)^2 \Delta^N\eta(x)^2
+\eta(x)^2 \Delta^N\eta(y)^2 \big\} K_N(x-y)\\
&\qquad+\frac{N^2}{N}\sum_{x\in\mathbb{T}_N}\big\{ S\big( \eta(x)^2 \eta(x+1)^2\big) 
K_N(1)
+S\big(\eta(x-1)^2 \eta(x)^2\big) K_N(-1)
+S\eta(x)^4 K_N(0) \big\}. 
\end{align*}
Since $K_N$ is symmetric, by adding and subtracting appropriate terms, the last display becomes
\begin{equation}
\label{eq:symmetric_part_eta^2}
\frac{2}{N} \sum_{x,y\in\mathbb{T}_N}\eta(x)^2 \eta(y)^2 \Delta_N K_N(x-y)
+2N\sum_{x\in\mathbb{T}_N}\big( \eta(x+1)^2-\eta(x)^2\big)^2 \big(K_N(0)-K_N(1)\big) .
\end{equation}
The leftmost term in the last display can be rewritten as 
\begin{equation*}
\begin{aligned}
&-\frac{2}{N}\sum_{x,y\in\mathbb{T}_N}
\eta(x)^2 \eta(y)^2 (I-\Delta_N)K_N(x-y)
+\frac{2}{N}\sum_{x,y\in\mathbb{T}_N}\eta(x)^2 \eta(y)^2 K_N(x-y) \\
&\quad=-\frac{2}{N}\sum_{x\in\mathbb{T}_N}\eta(x)^4 +2\|\eta^2\|^2_{-1,N}
\end{aligned}
\end{equation*}
where we used the identity \eqref{eq:H-1_kernel_property}. 
On the other hand, note that the rightmost term in \eqref{eq:symmetric_part_eta^2} can be rewritten, by using  \eqref{eq:H-1_kernel_identity_difference} with $x=0$, as
\begin{equation*}
\begin{aligned}
\frac{1}{N}\sum_{x\in\mathbb T_N} \big( \eta(x+1)^2-\eta(x)^2\big)^2 
\big( 1- K_N(0)\big). 
\end{aligned}
\end{equation*}
By the elementary identity $(a^2 - b^2)^2 \le a^4+b^4$, the last display is bounded by
\begin{equation*}
\frac{2(1-K_N(0))}{N}\sum_{x\in\mathbb T_N} \eta(x)^4. 
\end{equation*}
Hence, up to here we have that
\begin{equation}\label{eq:symmetric_part_eta^2_4th_moment_bd}
S_N \| \eta_t^2\|^2_{-1,N} 
\le -\frac{2K_N(0)}{N} \sum_{x\in\mathbb T_N} \eta_t(x)^4 
+ 2 \| \eta_t^2\|^2_{-1,N} .
\end{equation}
Now we turn to the contribution of the Hamiltonian part: 
\begin{equation*}
A(\eta(x)^2\eta(y)^2)
=2\big( \eta(x+1)-\eta(x-1)\big) \eta(y)^2 \eta(x) 
+2\big( \eta(y+1)-\eta(y-1)\big) \eta(x)^2 \eta(y) ,
\end{equation*}
and thus 
\begin{equation*}
\begin{aligned}
\alpha_N N^2 A \|\eta_t\|^2_{-1,N} 
&= \alpha_NN \sum_{x,y\in\mathbb T_N}K_N(x-y)
A\big( \eta(x)^2\eta(y)^2\big) \\
&= 4\alpha_NN \sum_{x,y\in\mathbb T_N}\big( \eta(x+1)-\eta(x-1)\big)  
\eta(y)^2 \eta(x) K_N(x-y) \\
&= 4\alpha_NN \sum_{x,y\in\mathbb T_N}\eta(x+1)\eta(x) \eta(y)^2 
\big( K_N(x-y) - K_N(x-y+1) \big)  
\end{aligned}
\end{equation*}
where in the second identity we used the symmetry of the kernel $K_N$. 
From \eqref{eq:H-1_kernel_identity_difference}, we can write the last term in last display as 
\begin{equation}
\label{eq:estimating_hamiltonian_eta2} 
\frac{2\alpha_N(1-K_N(0))}{N}\sum_{x,y\in\mathbb{T}_N}\eta(x)\eta(x+1)\eta(y)^2
-8\alpha_N N\sum_{\substack{x,y\in\mathbb{T}_N,\\ x>y}} \eta(x)\eta(x+1)\eta(y)^2 \sum_{j=1}^{x-y}K_N(j)
\end{equation}
where for the second sum in the last display we used again the fact that the kernel $K_N$ is symmetric. 
By the elementary inequality $2ab \le a^2 + b^2$ and the fact that we are summing over the torus, the first term in the last display is bounded by 
\begin{equation*}
2\big( 1-K_N(0) \big) \alpha_N N\Big(\frac{1}{N}\sum_{x\in\mathbb{T}_N}\eta(x)^2 \Big)^2.
\end{equation*}

On the other hand, by Young's inequality, the absolute value of the second sum in \eqref{eq:estimating_hamiltonian_eta2} is bounded above by 
\begin{equation*}
8\frac{\alpha_N}{N}\sum_{\substack{x,y\in\mathbb{T}_N, \\ x>y} }\eta(x)^2 \eta(y)^2 \Big\vert\sum_{j=1}^{x-y}K_N(j)\Big\vert
+8 \frac{\alpha_N}{N} \sum_{\substack{x,y\in\mathbb{T}_N, \\ x>y}} \eta(x+1)^2\eta(y)^2 \Big\vert\sum_{j=1}^{x-y}K_N(j)\Big\vert.
\end{equation*}
Now we estimate the letfmost term in the last display, but the rightmost can be treated analogously.
Recalling the definition of $K_N$,  we get that
\begin{equation*}
\bigg| \sum_{j=1}^{x-y}K_N(j) \bigg|
=\bigg| \frac{1}{N}\sum_{z\in\mathbb{T}_N}\frac{1}{a_N(z)} \sum_{j=1}^{x-y} e^{\frac{2\pi izj}{N}} \bigg| 
\le \frac{x-y}{N}  
\end{equation*}
for any $x,y$ such that $x>y$ where we used the fact that the sum $\sum_z 1/a_N(z)$ is bounded by a finite constant uniformly in $N$. 
Hence, the first term in the penultimate display, and thus the second sum of \eqref{eq:estimating_hamiltonian_eta2} is bounded by 
\begin{equation*}
\frac{16\alpha_N}{N}
\sum_{\substack{x,y\in\mathbb T_N, \\ x>y}}\eta(x)^2 \eta(y)^2 \frac{x-y}{N} 
\le 16\alpha_N N
\Big( \frac{1}{N}\sum_{x\in\mathbb T_N} \eta(x)^2 \Big)^2. 
\end{equation*}
Therefore, we conclude that there exists a constant $C>0$ such that 
\begin{equation*}
\begin{aligned}
L_N\|\eta^2\|^2_{-1,N}
&\le -2\frac{K_N(0)}{N} \sum_{x\in\mathbb{T}_N}\eta(x)^4 
+2\|\eta^2\|^2_{-1,N}
+C\alpha_N N\Big(\frac{1}{N}\sum_{x\in\mathbb{T}_N}\eta(x)^2 \Big)^2.
\end{aligned}
\end{equation*}
From Dynkin's formula, 
\begin{equation*}
\| \eta_t^2\|^2_{-1,N} 
- \| \eta_0^2 \|^2_{-1,N} 
- \int_0^t L_N \| \eta_s^2\|^2_{-1,N} ds 
\end{equation*}
is a mean-zero martingale. 
Taking the expectation in the last display and using the bound of the penultimate display, we have that 
\begin{equation*}
\begin{aligned}
\mathbb E_N\big[ \| \eta_t^2\|^2_{-1,N}\big]
&\le \mathbb E_N\big[ \| \eta_0^2\|^2_{-1,N}\big]
- \frac{2K_N(0)}{N} \int_0^t  \sum_{x\in\mathbb T_N} \mathbb E_N[\eta_s(x)^4] ds \\
&\quad+ \int_0^t \mathbb E_N\big[ \| \eta_s^2\|^2_{-1,N}\big]ds 
+ C\alpha_N N 
\end{aligned}
\end{equation*}
with some $C=C(T,\|e_0^N\|_{\ell^\infty(\mathbb T_N)})>0$.
Now, the proof ends as a consequence of Gronwall's inequality.

\section{Bounds on the two-point correlation function}
\label{sec:two-point_correlation_estimate}
Here let us give a proof of \cref{lem:correlation_estimate} on the two-point correlations (recall \eqref{eq:correlation_volume_definition} for the definition).  
From \cref{app:computation_generator} (see \cref{lem:action_correlation}), we know that 
\begin{equation*}
\begin{split}
\partial_t \varphi^N_t(x,y)
=\mathscr L_N^{(2)} 
\varphi_t^N(x,y)
+ \mathfrak g^N_t(x,y) 
\end{split}
\end{equation*} 
for any $(x,y)\in V_N^{(2)}$ where 
$
\mathfrak g^N_t(x,y)
= g^N_t(x)\mathbf{1}_{y=x+1}
+ g^N_t(x-1)\mathbf{1}_{y=x-1}
$
with
\begin{equation}
\label{eq:def of g}
\begin{aligned}
g^N_t(x) 
= \alpha_NN^2\big( e^N_t(x+1) -  v^N_t(x+1)^2 
- e^N_t(x) +  v^N_t(x)^2 \big) 
- (\nabla^N v_t^N(x))^2 ,
\end{aligned}
\end{equation}
and $\mathscr L_N^{(2)}$ denotes the operator that acts on functions $\varphi$ as 
\begin{equation}
\label{eq:2d_RW_operator_definition}
\begin{split}
\mathscr L_N^{(2)}\varphi(x,y)
=
&\big\{N^2(1+\alpha_N)\big(\varphi(x+1,y)+\varphi(x,y+1)-2 \varphi(x,y)\big)\\
&\quad + N^2(1-\alpha_N)\big(\varphi(x-1,y)+\varphi(x,y-1)-2 \varphi(x,y)\big) \big\} \textbf{1}_{|y-x|>1} \\
&+\big\{ N^2(1+\alpha_N)\big(\varphi(x,x+2)- \varphi(x,x+1)\big) \\
&\quad + N^2(1-\alpha_N)\big(\varphi(x-1,x+1)- \varphi(x,x+1)\big)\big\}\textbf{1}_{y=x+1}\\ 
&+\big\{ N^2(1+\alpha_N)\big(\varphi(x+1,x-1)- \varphi(x,x-1)\big) \\
&\quad+ N^2(1-\alpha_N)\big(\varphi(x,x-2)-\varphi(x,x-1)\big)\big\}\textbf{1}_{y=x-1}.
\end{split}
\end{equation} 
From Duhamel's principle, $\varphi^N_t$ admits the following probabilistic representation: 
\begin{equation}
\label{eq:duhamel_principle}
\varphi^N_t(x,y)
= \mathbf{E}_{(x,y)}\bigg[
\varphi^N_0(X^{(2)}_t) 
+ \int_0^t \mathfrak g^N_{t-s}(X^{(2)}_s) ds
\bigg] 
\end{equation}
where $\{ X^{(2)}_t:t\ge0 \}$ is the random walk on $V^{(2)}_N$ with generator $\mathscr L^{(2)}_N$, and this is indeed justified from the fact that the operator $\mathscr L^{(2)}_N$ is a bounded operator. We denote by $\mathbf P_v$ the corresponding probability measure associated to this random walk starting from a given position $v \in V^{(2)}_N$ and by $\mathbf E_v$ the expectation with respect to $\mathbf P_v$. 
Note that the random walk $\{ X^{(2)}_t:t\ge0 \}$ jumps up/right (resp. left/down) with rate $N^2(1+\alpha_N)$ (resp.  $N^2(1-\alpha_N)$), except  at the diagonal lines
$$
\mathcal D_\pm^{(2)} \coloneqq 
\{ (x,x\pm 1) ; x\in \mathbb T_N \} , 
$$ 
that is, e.g.  on the line $\mathcal D_+^{(2)}$ the random walk only jumps left (resp. right) with rate $N^2(1-\alpha_N)$ (resp.  $N^2(1+\alpha_N)$). Analogously, the random walk is also reflected on the line $\mathcal D_-^{(2)}$.  

From this it follows that
\begin{equation*}
\begin{split}
\varphi^N_t(x,y)
&= \mathbf{E}_{(x,y)}\big[\varphi^N_0(X^{(2)}_t)\big]
+ \int_0^t \sum_{(x',y')\in \mathbb T_N^2} 
\mathfrak g^N_{t-s}(x',y')  
\mathbf{P}_{(x,y)}(X^{(2)}_s=(x',y')) ds . 
\\&= \mathbf{E}_{(x,y)}\big[\varphi^N_0(X^{(2)}_t)\big] + \int_0^t \sum_{x'\in \mathbb T_N} g^N_{t-s}(x') \mathbf{P}_{(x,y)}(X^{(2)}_s=(x',x'+1)) ds  \\
&\quad+ \int_0^t \sum_{x'\in \mathbb T_N}  g^N_{t-s}(x'-1) \mathbf{P}_{(x,y)}(X^{(2)}_s=(x',x'-1)) ds . 
\end{split}
\end{equation*} 
Hence, we have the bound 
\begin{equation*}
\begin{aligned}
\| \varphi^N_t\|_{\ell^\infty(V_N^{(2)})}
\le \| \varphi^N_0\|_{\ell^\infty(V_N^{(2)})}
+ \sup_{0\le s \le t}\| g^N_s\|_{\ell^\infty(\mathbb T_N)} 
\max_{(x,y)\in V_N^{(2)}} 
T^N_t(x,y) 
\end{aligned}
\end{equation*}
where 
\begin{equation*}
T^N_t(x,y)
= \int_0^t \mathbf{P}_{(x,y)}(X^{(2)}_s \in \partial V_N^{(2)}) ds
\end{equation*}
is the total time that the random walk stays at the  boundary set $\partial V_N^{(2)}\coloneqq\mathcal D_+^{(2)} \cup \mathcal D_-^{(2)}$ during the time interval $[0,t]$ starting from $(x,y)$.
Now, we have the following two estimates.

\begin{lemma}
\label{lem:two-point_local_time_estimate}
There exists a constant $C=C(T)>0$ such that  
\begin{equation*}
\sup_{0\le t\le T}
\max_{(x,y)\in V_N^{(2)}} |T_t^N(x,y)| \le C/N.
\end{equation*}
\end{lemma}

\begin{lemma}
\label{lem:two-point_duhamel_reminder_estimate}
Assume \eqref{eq:assumption_initial_measure_volume} and \eqref{eq:assumption_initial_meaure_energy}.
Then, there exists another constant $C=C(T, v_0,e_0)>0$ such that 
\begin{equation}
\label{eq:correlation_estimate_key_lemma}
\begin{aligned}
\sup_{0\le t\le T} \| g^N_t\|_{\ell^\infty(\mathbb T_N)}
\le 
C\alpha_N^2N^2\Big( 1+ N^{2/3} \sup_{0\le t\le T} \| \varphi^N_t \|_{\ell^\infty(V_N^{(2)})} \Big) .
\end{aligned}
\end{equation}
\end{lemma}

Given the last two lemmas, the proof of the estimate of two-point correlations is straightforward.
Indeed, we have that
\begin{equation*}
\begin{aligned}
\sup_{0\le t\le T} \|\varphi^N_t\|_{\ell^\infty(V_N^{(2)})} 
&\le 
\|\varphi^N_0\|_{\ell^\infty(V_N^{(2)})}  
+\frac{C}{N}\sup_{0\le t \le T} \| g^N_t\|_{\ell^\infty(\mathbb T_N)} \\ 
&\le \|\varphi^N_0\|_{\ell^\infty(V_N^{(2)})}  
+ C\alpha_N^2 N^{5/3} 
\sup_{0\leq t\leq T}\|\varphi_t^N\|_{\ell^\infty(V_N^{(2)})} +C\alpha_N^2N 
\end{aligned}
\end{equation*}
with some $C>0$. 
Thus, by the assumption \eqref{eq:assumption_initial_measure_correlation}, we have that 
\begin{equation}
\big(1-C\alpha_N^2 N^{5/3}\big)\sup_{0\leq t\leq T}\|\varphi_t^N\|_{\ell^\infty(V_N^{(2)})}
\le C'\Big( \alpha_N+\frac{1}{N} \Big)
\end{equation} 
with some $C,C'>0$. 
When $\kappa\ge 1>5/6$, for $N$ sufficiently large, we can absorb the second term on the right-hand side of the penultimate display, and thus we complete the proof of the desired bound of the two-point correlations. 

The proof of \cref{lem:two-point_local_time_estimate} is postponed to \cref{sec:rw_estimates}.
Here let us give a proof of \cref{lem:two-point_duhamel_reminder_estimate}. 

\begin{proof}[Proof of \cref{lem:two-point_duhamel_reminder_estimate}]
Recalling the definition of $g^N_t(x)$ in \eqref{eq:def of g}, note that 
\begin{equation*}
g^N_t(x) 
= \alpha_NN^2 h^N_t(x) 
- (\nabla^N v^N_t(x))^2 
\end{equation*}
where we set 
\begin{equation*}
h^N_t(x)
= e^N_t(x+1)-e^N_t(x) - v^N_t(x+1)^2 + v^N_t(x)^2.
\end{equation*}
Since $ v^N_t$ satisfies the discrete heat equation and under the assumption \eqref{eq:assumption_initial_measure_volume}, it is easy to see that the last term in the last display is bounded by a constant that depends only on $T$, so that our task is to give an estimate of $\|h^N_t\|_{\ell^\infty(\mathbb T_N)}$. 
A simple computation shows that
\begin{equation*}
\begin{split} 
\partial_t h_t^N(x)
=\Delta^{(1)}_N h^N_t(x)
+ \mathfrak h^N_t(x)
\end{split}
\end{equation*}
where
\begin{equation*}
\begin{aligned}
\mathfrak h^N_t(x) 
&=2\alpha_N \Delta^{(1)}_N \varphi^N_t(x,x+1) 
+ (\nabla^N v^N_t(x+1))^2 - (\nabla^N v^N_t(x-1))^2 .  
\end{aligned}
\end{equation*}
Now from Duhamel's principle, we get 
\begin{equation*}
\begin{aligned}
h_t^N(x)
&=E_x\bigg[ h_0^N(Y_t^{(1)})
+ \int_0^t \mathfrak h^N_{t-s}(Y_s^{(1)})ds \bigg] \\
&=E_x\big[ h_0^N(Y_t^{(1)}) \big] 
+ \int_0^t \sum_{y\in\mathbb T_N} \mathfrak h^N_{t-s}(y) P_x(Y_s^{(1)}=y) ds 
\end{aligned}
\end{equation*}
where recall that $\{Y_t^{(1)}: t\ge 0\}$ denotes the simple symmetric random walk on $\mathbb T_N$ with rate $N^2$ to all possible directions, starting from 
$x$, under the corresponding measure $P_x$ and we denote by $E_x$ the expectation with respect to the measure $P_x$. 

Now, recalling the definition of $\mathfrak h^N_t$, we have that 
\begin{equation}
\label{eq:two-point_h-function_after_duhamel}
\begin{aligned} 
h_t^N(x)
&=E_x\big[ h_0^N(Y_t^{(1)}) \big] 
+2\alpha_N \int_0^t \sum_{y\in\mathbb T_N} 
\varphi^N_{t-s}(y,y+1) \Delta_N^{(1)} P_x(Y_s^{(1)}=y) ds \\
&\quad+ \int_0^t \sum_{y\in\mathbb T_N}
\big[(\nabla^N v^N_s(y+1))^2 - (\nabla^N v^N_s(y-1))^2 \big] P_x(Y_s^{(1)}=y) ds
\end{aligned}
\end{equation}
where we used a summation-by-parts for the second term of the right-hand side and note that the discrete Laplacian is acting on the variable $y$. 
Note that the absolute value of the third term on the right-hand side of the last display is bounded from above by 
\begin{equation*}
\begin{aligned}
\frac{C}{N} \int_0^t \sum_{y\in\mathbb T_N} P_x(Y_s^{(1)}=y) ds 
\le \frac{Ct}{N}  
\end{aligned}
\end{equation*}
for some $C=C(\| v_0'\|_{L^\infty(\mathbb T)}, \|  v_0''\|_{L^\infty(\mathbb T)})>0$ where we used \cref{lem:pde_estimate_laplacian_of_rho_squared} below.
This gives a term of order $O(\alpha_N^2N^2)$ in the right-hand side of the desired estimate \eqref{eq:correlation_estimate_key_lemma}.  
To proceed, let us define 
\begin{equation*}
\mathcal I^N_{[a,b)}
= \int_{a}^b \sum_{y\in\mathbb T_N}\varphi^N_{t-s}(y,y+1)
\Delta_N^{(1)} P_x(Y_s^{(1)}=y) ds
\end{equation*}
for $a,b\in [0,T]$ such that $a<b$. 
Hereafter we handle the second term in the right-hand side of \eqref{eq:two-point_h-function_after_duhamel}, that is, $2\alpha_N\mathcal I^N_{[0,t)}$. 
To that end, fix $t_0\in[0,t]$ and we split the integral according to short times and long times. 
First, for the integral over the short interval $[0,t_0)$, we have the bound 
\begin{equation*}
\begin{split}
| \mathcal I^N_{[0,t_0)} | 
\le \sup_{0\le s\le t} \| \varphi^N_s\|_{\ell^\infty(V_N^{(2)})} 
\int_{0}^{t_0} \sum_{y\in\mathbb T_N} \big|\Delta_N^{(1)} P_x(Y_s^{(1)}=y)\big| ds 
\le CN^2 t_0 \sup_{0\le s\le t} \| \varphi^N_s\|_{\ell^\infty(V_N^{(2)})}
\end{split}
\end{equation*}
with some $C>0$, where we used the fact that $P_x$ is a probability.
On the other hand, for the integral over the other interval $[t_0,t)$, we use a global estimate of the transition probability as follows (see \cref{lem:SRW_transition_prob_estimate}): 
\begin{equation*}
|\Delta_N^{(1)} P_x(Y_t^{(1)}=y)|
\le \frac {C}{Nt^{3/2}},
\end{equation*}
for any $x,y\in\mathbb T_N$ and for any $t>0$.  
Then, we have that
\begin{equation*}
\begin{split}
| \mathcal I^N_{[t_0,t)]}| 
&\le \sup_{0\le s\le t} \| \varphi^N_s\|_{\ell^\infty(V_N^{(2)})} 
\int_{t_0}^t \sum_{y\in\mathbb T_N} 
\big| \Delta_N^{(1)}P(Y_s^{(1)}=y)\big| ds \\
&\le \sup_{0\le s\le t} \| \varphi^N_s\|_{\ell^\infty(V_N^{(2)})}
\int_{t_0}^t \frac{ds}{s^{3/2}}
\le \frac{2}{\sqrt{t_0}} \sup_{0\le s\le t} \| \varphi^N_s\|_{\ell^\infty(V_N^{(2)})}. 
\end{split}
\end{equation*}
Now we optimize in $t_0$ to get $t_0=N^{-4/3}$ and 
\begin{equation*}
\begin{split}
| \mathcal I^N_{[0,t)} | 
\le CN^{2/3} \sup_{0\le s\le t} \| \varphi^N_s\|_{\ell^\infty(V_N^{(2)})}
\end{split}
\end{equation*}
for some $C>0$. 
Hence, going back to \eqref{eq:two-point_h-function_after_duhamel}, we have the bound 
\begin{equation*}
\begin{aligned}
\sup_{0\le t\le T} \| h^N_t\|_{\ell^\infty(\mathbb T_N)}
&\le \| h^N_0\|_{\ell^\infty(\mathbb T_N)}
+ C \alpha_N \big| \mathcal I^N_{[0,T)} \big| 
+ \frac{CT}{N}\\
&\le \| h^N_0\|_{\ell^\infty(\mathbb T_N)}
+ C \alpha_NN^{2/3} 
\sup_{0\le t\le T} \| \varphi^N_t\|_{\ell^\infty(V_N^{(2)})}
+ \frac{CT}{N} .
\end{aligned}
\end{equation*}
This immediately completes the proof of the desired assertion, noting that the assumption \eqref{eq:assumption_initial_meaure_energy} enables us to see that $\| h^N_0\|_{\ell^\infty(\mathbb T_N)} \le C/N$ for some $C=C(e_0)>0$. 
\end{proof}

\begin{lemma}
\label{lem:pde_estimate_laplacian_of_rho_squared}
Assume \eqref{eq:assumption_initial_measure_volume}. 
For any $t\ge 0$, there exists some constant $C>0$ depending on $t$ and the initial volume profile $v^N_0(x)$ such that 
\begin{equation*}
\max_{x\in\mathbb T_N} 
N\big|
(\nabla^N v^N_t(x+1))^2 - (\nabla^N v^N_t(x-1))^2 \big|
\le C. 
\end{equation*}
\end{lemma}
\begin{proof}
First, note that we have an identity 
\begin{equation*}
\begin{aligned}
&(\nabla^N v^N_t(x+1))^2 - (\nabla^N v^N_t(x-1))^2 \\
&\quad
= \big( \nabla^N v^N_t(x+1) - \nabla^N v^N_t(x-1)\big)
\big( \nabla^N v^N_t(x+1))+ (\nabla^N v^N_t(x-1)\big) .
\end{aligned}
\end{equation*}
Note that the sum of gradient has a uniform bound at any time, according to a simple argument using a maximum principle. 
Now, our task is to show that $2\widetilde{\nabla}^N\nabla^N v^N_t(x)$ is uniformly bounded at any $t$.
To that end, notice that 
\begin{equation*}
\begin{aligned}
2\widetilde{\nabla}^N\nabla^N v^N_t(x)
&= N^2
\big[v^N_t(x+2)-v^N_t(x+1) - v^N_t(x) + v^N_t(x-1)\big]\\
&= N^2 \big[ 
\big(v^N_t(x+2)-v^N_t(x+1)\big) 
- \big(v^N_t(x+1) - v^N_t(x) \big) \\
&\quad+ \big( v^N_t(x+1) - v^N_t(x)\big) 
- \big( v^N_t(x) - v^N_t(x-1)\big) 
\big] \\
&= \Delta^{(1)}_N v^N_t(x+1) 
+ \Delta^{(1)}_N v^N_t(x) . 
\end{aligned}
\end{equation*}
Hence, again by the maximum principle, we can easily check that 
\begin{equation*}
\| \Delta_N^{(1)}v^N_t\|_{\ell^\infty(\mathbb T_N)} \le C
\end{equation*} 
for some $C>0$, provided the same bound holds at time $0$. Hence, we complete the proof of the assertion. 
\end{proof}

\appendix
\section{Evolution of the two-point correlation function}
\label{app:computation_generator} 
Recall the operator $\mathscr L^{(2)}_N$ defined in \eqref{eq:2d_RW_operator_definition}. 
Here let us compute the derivative in time of the correlation function.  
Using Kolmogorov's forward equation, we get that
\begin{equation}
\label{eq:derivative_volume_correlation}
\begin{aligned}
\partial_t\varphi_t^N(x,y)
=\mathbb E_N\big[L_N\big(\eta_t(x)\eta_t(y)\big)\big] -
\partial_t\big( v_t^N(x) v_t^N(y)\big) 
\end{aligned}
\end{equation}
for each $x,y\in\mathbb Z$. 
A long but straightforward computation enables us to compute the time evolution of the two-point correlation function as follows. 

\begin{lemma}
\label{lem:action_correlation}
For each $x,y\in V_N^{(2)}$, we have that 
\begin{equation}\label{eq:time_evolution_volume_correlation}
\begin{split}
\partial_t \varphi^N_t(x,y)
=\mathscr L_N^{(2)} 
\varphi_t^N(x,y)
+ \mathfrak g^N_t(x,y) 
\end{split}
\end{equation} 
where 
$
\mathfrak g^N_t(x,y)
= g^N_t(x)\mathbf{1}_{y=x+1}
+ g^N_t(x-1)\mathbf{1}_{y=x-1}
$
and $g^N_t(x)$ is defined in \eqref{eq:def of g}. 
\end{lemma}

Here, note that since \eqref{eq:time_evolution_volume_correlation} is a system of finite numbers of ordinary differential equations, given an initial correlation function $\varphi^N_0$, it is clear that the system has a unique solution. 
To compute the action of the generator on the correlation function, we split the computation into two cases: (i) $|x-y|>1$ and (ii) $|x-y|=1$. 

\subsection*{(i) The case \texorpdfstring{$|x-y|>1$}{in bond}}
First, note that 
\begin{align*}
L_N\big(\eta(x)\eta(y)\big) 
&= \eta(x) (\Delta^N\eta(y)) + (\Delta^N\eta(x))\eta(y) \\
&\quad+ \alpha_NN^2 \big\{(\eta(x+1)-\eta(x-1))\eta(y)+(\eta(y+1)-\eta(y-1))\eta(x)\big\},
\end{align*}
so that 
\begin{equation}
\label{eq:expectation_action_correlation}
\begin{aligned}
&\mathbb E_N\big[L_N(\eta(x)\eta(y))\big] \\
&\quad= N^2 \mathbb E_N\big[ \eta(y)\eta(x+1)+\eta(y)\eta(x-1)+\eta(x)\eta(y+1)+\eta(x)\eta(y-1)-4\eta(x)\eta(y)\big] \\
&\qquad+ \alpha_NN^2 \mathbb E_N\big[ \eta(x+1)\eta(y)-\eta(x-1)\eta(y)+\eta(y+1)\eta(x)-\eta(y-1)\eta(x)\big] . 
\end{aligned}
\end{equation}
After centering, we can simplify the relation \eqref{eq:expectation_action_correlation} as 
\begin{equation}
\label{eq:expectation_ection_correlation_computed}
\begin{aligned}
&\E_N\big[ L_N(\eta(x)\eta(y))\big] \\
&\quad= \Delta^{(2)}_N\varphi_t^N(x,y)
-\alpha_NN^2\Big\{\varphi_t^N(x,y-1)+\varphi_t^N(x-1,y)-\varphi_t^N(x,y+1)-\varphi_t^N(x+1,y)\Big\}\\
&\qquad+  v_t^N(y)\Delta^N v_t^N(x)
+ v_t^N(x)\Delta^N v_t^N(y)\\
&\qquad+ \alpha_NN^2 \Big\{ v_t^N(x+1) v_t^N(y)- v_t^N(x-1) v_t^N(y)+ v_t^N(x) v_t^N(y+1)- v_t^N(x) v_t^N(y-1)\Big\} .
\end{aligned}
\end{equation}
To compute the term with the time derivative in \eqref{eq:derivative_volume_correlation}, note that 
\begin{align*}
\partial_t\big( v_t^N(x) v_t^N(y)\big)
&=  v_t^N(y)\partial_t v_t^N(x) 
+  v_t^N(x)\partial_t v_t^N(y) \\
&=  v_t^N(y)\E_N[L_N\eta(x)]+  v_t^N(x)\E_N[L_N\eta(y)] 
\end{align*}
again by the Kolmogorov equation. 
Since 
\begin{equation*}
L_N\eta(x) 
=\Delta^N \eta(x) 
+ \alpha_NN^2 \big( \eta(x+1)-\eta(x-1)\big) ,
\end{equation*}
we can deduce the time evolution of the volume element \eqref{eq:discrete_hdl}, that is,  
\begin{equation*}
\partial_t v_t^N(x)
=\Delta^N  v_t^N(x)
+\alpha_NN^2 \big(  v^N_t(x+1)- v^N_t(x-1)\big).
\end{equation*}
Then, the second term in \eqref{eq:derivative_volume_correlation} becomes 
\begin{align*}
& v_t^N(y)\Delta^N v_t^N(x)
+ v_t^N(x)\Delta^N v_t^N(y)\\
&+\alpha_NN^2 \Big\{ v_t^N(x+1) v_t^N(y)- v_t^N(x-1) v_t^N(y)+ v_t^N(x) v_t^N(y+1)- v_t^N(x) v_t^N(y-1)\Big\},
\end{align*}
which is nothing but the last two terms in \eqref{eq:expectation_ection_correlation_computed}. 
Therefore, for $|x-y|>1$,
\begin{equation*}
\partial_t\varphi_t^N(x,y)
=\Delta^{(2)}_N\varphi_t^N(x,y)
-\alpha_NN^2\Big\{\varphi_t^N(x,y-1)+\varphi_t^N(x-1,y)-\varphi_t^N(x,y+1)-\varphi_t^N(x+1,y)\Big\},
\end{equation*}
which can be written as 
\begin{equation*}
\partial_t\varphi_t^N(x,y)=\mathscr L^{(2)}_N\varphi_t^N(x,y).
\end{equation*}

\subsection*{(ii) The case \texorpdfstring{$|x-y|=1$}{nearest-neighbor}} 
We present the computation for the case $y=x+1$. The remaining case $y=x-1$ follows by symmetry. First, note that 
\begin{equation*}
\begin{aligned}
L_N(\eta(x)\eta(x+1))
&= \eta(x) (\Delta^N\eta(x+1)) + (\Delta^N\eta(x))\eta(x+1) 
- (\nabla^N \eta(x))^2 \\
&\quad+ \alpha_NN^2 \big\{(\eta(x+1)-\eta(x-1))\eta(x+1)+(\eta_{x+2}-\eta_{x})\eta(x)\big\},
\end{aligned}
\end{equation*}
where the third term in the right-hand side of the last display is a correction term by the \textit{carr\'e du champ} operator, and the rest of the computation is not varying except for this correction term.  
Thus, we have that 
\begin{equation*}
\begin{aligned}
\partial_t\varphi_t^N(x,x+1)
&=\Delta^{(2)}_N\varphi_t^N(x,x+1)
- \mathbb E_N[(\nabla^N \eta(x))^2]\\
&-\alpha_NN^2\Big\{\varphi_t^N(x,x)+\varphi_t^N(x-1,x+1)-\varphi_t^N(x,x+2)-\varphi_t^N(x+1,x+1)\Big\} \\
&=\mathscr L^{(2)}_N\varphi_t^N(x,x+1)+N^2\Big(\varphi_t^N(x+1,x+1)+\varphi_t^N(x,x)-2\varphi_t^N(x,x+1)\Big)\\
&\quad-\alpha_N N^2\Big(\varphi_t^N(x,x)-\varphi_t^N(x+1,x+1)\Big)\\
&\quad- N^2e^N_t(x) 
- N^2 e^N_t(x+1)
+ 2N^2\big(\varphi^N_t(x,x+1)+ v^N_t(x) v^N_t(x+1)\big) \\
&=\mathscr L^{(2)}_N\varphi_t^N(x,x+1)-N^2\Big(v_t^N(x+1)-v_t^N(x)\Big)^2\\
&\quad -\alpha_N N^2\Big(e^N(x)-v_t^N(x)^2-e^N_t(x+1)+v_t^N(x+1)^2\Big).
\end{aligned}
\end{equation*}
Therefore, we conclude that 
\begin{equation*}
    \partial_t\varphi_t^N(x,x+1)=\mathscr L^{(2)}_N\varphi_t^N(x,x+1)+g_t^N(x).
\end{equation*}

\section{Random walk estimates}
\label{sec:rw_estimates}
Recall that the estimate on correlation function was governed by the ones for the random walk $\{X^{(2),N}_t: t \ge0\}$ on $V^{(2)}_N$ generated by $\mathscr L_N^{(2)}$.
Here we give a quantitative estimate associated to random walks presented in the main part of the paper. 
Throughout this section, let us omit the dependency on $N$ for symbols concerning random walks.  
First, let us give a proof of \cref{lem:two-point_local_time_estimate}. 
To that end, we make the following mapping of the two-dimensional random walk $\{X_t^{(2)}: t\ge0\}$ with a one-dimensional random walk $\{\widetilde X_t^{(1)}: t\ge0\}$ on the space $\Lambda_{N-1}=\{1,\ldots,N-1\}$ in the following way. We map the diagonal lines $y=x+r$ with $r\in\Lambda_{N-1}$. 
Since the two-dimensional random walk moves along the diagonals $y=x+r$ to $y=x+r\pm1$ with the same rates $2N^2$, then the map gives us a one-dimensional random walk which moves in $r\in\Lambda_{N-1}$ with rates $2N^2$ but at $x=1$ and $x=N-1$ it only jumps to $x=2$ and $x=N-2$ respectively. 
The generator of this projected random walk is given at $x\in\{1,\cdots, N-1\}$ by 
\begin{equation}
\label{eq:1d_RW_after_projection_generator}
\mathfrak L_N^{(1)}
f(x)
= 
\begin{cases}
\begin{aligned}
& 2N^2 \big( f(x+1)+f(x-1) - 2f(x) \big) && \text{ if } x\in\{2,\ldots, N-2\} ,\\
& 2N^2\big( f(2)-f(1)\big) && \text{ if } x=1, \\
& 2N^2\big( f(N-2)-f(N-1)\big)  && \text{ if } x=N-1. 
\end{aligned}
\end{cases}
\end{equation}
for any test function $f$. 
Then, our task will be to estimate the local time where this projected random walk $\{\widetilde X_t^{(1)}: t \ge 0\}$, generated by the operator $\mathfrak L_N^{(1)}$, stays at the boundary points $x=1, N-1$ up to time $t$.
This is more tractable than estimating the original two-dimensional random walk.  
In what follows, let us denote by $\mathbf P_x$ the probability associated to the projected random walk $\{\widetilde X_t^{(1)}:t\ge 0\}$ starting from $x\in \Lambda_{N-1}$ and write the expectation with respect to $\mathbf P_x$ by $\mathbf E_x$. 

Now, let us justify the above strategy as follows.  
Recalling that $\{X^{(2)}_t:t\ge0\} $ is a random walk generated by $\mathscr L^{(2)}_N$, it follows that $\mathbf{P}_{(x,y)}(X_t^{(2)} =(x',y'))$ satisfies the discrete equation: 
\begin{equation*}
\begin{cases}
\begin{aligned}
&\partial_t \mathbf{P}_{(x,y)}(X_t^{(2)} =(x',y'))
= \mathscr L^{(2)}_N \mathbf{P}_{(x,y)}(X_t^{(2)} =(x',y')),  \\
& \mathbf{P}_{(x,y)}(X_0^{(2)} =(x',y'))
= \mathbf{1}_{(x,y)=(x',y')}
\end{aligned}
\end{cases}
\end{equation*}
where the operator $\mathscr L^{(2)}_N$ acts on the variable $(x',y')$ in the evolution equation. 
To go through further, set $t\in\{ 1,\ldots, N-1\}$ via
$r \equiv y-x \pmod N$.
Moreover, for any $r' \in \{ 0,\ldots, N-1\}$, let 
\begin{equation*}
p_t(r,r')
= \sum_{\substack{(x',y')\in V_N^{(2)}, \\ y'-x'\equiv r'}} 
\mathbf{P}_{(x,y)}(X_t^{(2)} =(x',y')) 
\end{equation*}
where the condition in the summation symbol is modulo $N$. 
Then, we have that 
\begin{equation*}
p_0(r,r') 
= \sum_{\substack{(x',y')\in V_N^{(2)}, \\ y'-x'\equiv r'}}
\mathbf{1}_{(x,y)=(x',y')}
= \mathbf{1}_{r=r'} .
\end{equation*}
Moreover, for $(x',y') \in V_N^{(2)}\setminus \partial V_N^{(2)}$, namely when $r'$ is in the bulk $\{ 2,\ldots,N-2\}$ , we have 
\begin{equation*}
\begin{aligned}
\partial_t p_t(r,r')
&= \sum_{\substack{(x',y')\in V_N^{(2)}, \\ y'-x'\equiv r'}}
\mathscr L^{(2)}_N \mathbf{P}_{(x,y)}(X_t^{(2)} =(x',y')) \\ 
&= N^2(1+\alpha_N) 
\sum_{\substack{(x',y')\in V_N^{(2)}, \\ y'-x'\equiv r'}}
\Big(
\mathbf{P}_{(x,y)}(X_t^{(2)} =(x'+1,y'))
+ \mathbf{P}_{(x,y)}(X_t^{(2)} =(x',y'+1)) \\ 
&\qquad\qquad\qquad\qquad\qquad- 2 \mathbf{P}_{(x,y)}(X_t^{(2)} =(x',y')) \Big) \\
&\quad+  N^2(1-\alpha_N) 
\sum_{\substack{(x',y')\in V_N^{(2)}, \\ y'-x'\equiv r'}}
\Big(
\mathbf{P}_{(x,y)}(X_t^{(2)} =(x'-1,y'))
+ \mathbf{P}_{(x,y)}(X_t^{(2)} =(x',y'-1)) \\
&\qquad\qquad\qquad\qquad\qquad- 2 \mathbf{P}_{(x,y)}(X_t^{(2)} =(x',y')) \Big) \\
&= N^2(1+\alpha_N)
\big( p_t(r,r'-1) + p_t(r,r'+1) - 2p_t(r,r') \big) \\
&\quad+ N^2(1-\alpha_N)
\big( p_t(r,r'+1) + p_t(r,r'-1) - 2p_t(r,r') \big) \\
&= 2\Delta_N^{(1)} p_t(r,r')
\end{aligned}
\end{equation*}
where in the utmost right-hand side, the one-dimensional Laplacian $\Delta_N^{(1)}$ is acting on the variable $r'$. 
On the other hand, for the boundary case, we have that 
\begin{equation*}
\begin{aligned}
\partial_t p_t(r,1)
&= \sum_{\substack{(x',y')\in V_N^{(2)}, \\ y'-x'\equiv 1}}
\mathscr L^{(2)}_N \mathbf{P}_{(x,y)}(X_t^{(2)} =(x',y')) \\ 
&= \sum_{x'\in \mathbb T_N} 
\mathscr L^{(2)}_N \mathbf{P}_{(x,y)}(X_t^{(2)} =(x',x'+1)) \\ 
&= N^2(1+\alpha_N) \sum_{x'\in \mathbb T_N} 
\Big( \mathbf{P}_{(x,y)}(X_t^{(2)} =(x',x'+2))
- \mathbf{P}_{(x,y)}(X_t^{(2)} =(x',x'+1)) \Big)
\\ 
&\quad+ N^2(1-\alpha_N) \sum_{x'\in \mathbb T_N} 
\Big( \mathbf{P}_{(x,y)}(X_t^{(2)} =(x'-1,x'+1)) 
- \mathbf{P}_{(x,y)}(X_t^{(2)} =(x',x'+1)) \Big)
\\ 
&= N^2(1+\alpha_N) \big( p_t(r,2) - p_t(r,1)\big) 
+ N^2(1-\alpha_N) \big( p_t(r,2) - p_t(r,1)\big) \\
&= 2N^2 \big( p_t(r,2) - p_t(r,1)\big). 
\end{aligned}
\end{equation*}
Analogously, we have that 
\begin{equation*}
\partial_t p_t(r,N-1)
= 2N^2 \big( p_t(r,N-1) - p_t(r,N-1)\big). 
\end{equation*}
Therefore, it turns out that $p_t(r,r')$, $r,r'\in\Lambda_{N-1}$ satisfies 
\begin{equation*}
\partial_t p_t(r,r')
= \mathfrak L_N^{(1)} p_t(r,r')
\end{equation*}
where the operator $\mathfrak L_N^{(1)}$ is acting on the variable $r'$, so that $p_t(\cdot,\cdot)$ is nothing but the transition probability associated to the projected random walk $\{\widetilde X_t^{(1)}: t \ge0\}$:
\begin{equation*}
p_t(x,y)
= \mathbf P_x(\widetilde X_t^{(1)}=y) 
\end{equation*}
for any $x,y\in \Lambda_{N-1}$. 

With the help of the above reduction of random walks to a lower-dimensional space, we can show the estimate of the local time $T^N_t$. Note that we have the identity 
\begin{equation*}
\begin{aligned}
T^N_t(x,y)
= \int_0^t \big( p_s(r,1) + p_s(r, N-1)\big) ds
\end{aligned}
\end{equation*} 
where recall that $r\in \{ 1,\ldots, N-1\}$ is defined by $r\equiv y-x \pmod{N}$.  
Now, we are in a position to give a proof of \cref{lem:two-point_local_time_estimate}.

\begin{proof}[Proof of \cref{lem:two-point_local_time_estimate}]
Let us take a function $g(x)=-(x-x_0)^2$ where $x_0$ will be chosen later.  
A simple computation shows that 
\begin{equation}
\label{eq:RW_lap}
\begin{cases}
\mathfrak L_N^{(1)} g(x)= -2N^2 , \quad x\in \{2,\ldots, N-2\} , \\ 
\mathfrak L_N^{(1)} g(1)=-2N^2(3-2x_0), \\
\mathfrak L_N^{(1)} g(N-1)=-2N^2(-2N+3-2x_0).
\end{cases}
\end{equation}
Applying Dynkin's formula with the function $g$ we get 
\begin{equation*}
g(\widetilde X_t^{(1)})-g(\widetilde X_{0}^{(1)})
-\int_{0}^t \mathfrak L_N^{(1)} 
g(\widetilde X_s^{(1)})ds
\end{equation*}
is a mean-zero martingale with respect to the natural filtration. 
Thus, by taking the expectation in the last display, i.e., the definition of the mean-zero martingale, we get for any $x\in\Lambda_{N-1}$ that 
\begin{equation*}
0=
\mathbf E_{x}[g(\widetilde X_{t}^{(1)})]
-\mathbf E_{x}[g(\widetilde X_{0}^{(1)})]
- \mathbf E_{x}\bigg[\int_{0}^t \mathfrak L_N^{(1)} 
g(\widetilde X_{s}^{(1)})ds \bigg]
\end{equation*}
Now we use the computation \eqref{eq:RW_lap} to get 
\begin{equation*}
\begin{split}
\mathbf E_{x}\bigg[\int_{0}^t \mathfrak L_N^{(1)}g(\widetilde X_s^{(1)}) ds\bigg]
&=-2N^2 \int_0^t\mathbf P_{x}(\widetilde X_{s}^{(1)}\neq 1, N-1) ds
-2N^2(3-2x_0) \int_0^t\mathbf P_{x}(\widetilde X_{s}^{(1)}= 1) ds\\
&\quad-2N^2(-2N+3-2x_0)\int_0^t\mathbf P_{x}(\widetilde X_{s}^{(1)}= N-1) ds. 
\end{split}
\end{equation*}
If we choose $x_0=N/2+3/2$, then last identity becomes
\begin{equation*}
\begin{split}
\mathbf E_{x}\bigg[\int_{0}^t \mathfrak L_N^{(1)} 
g(\widetilde X_{s}^{(1)}) ds\bigg]
&=-2N^2 \int_0^t\mathbf P_{x}(\widetilde X_{sN^2}\neq 1, N-1) ds
+2N^3\int_0^t\mathbf P_{x}(\widetilde X_{s}^{(1)}= 1) ds\\
&\quad+6N^3\int_0^t\mathbf P_{x}(\widetilde X_{s}^{(1)}= N-1) ds.
\end{split}
\end{equation*}
Hence, it follows that   
\begin{equation*}
\begin{split}
2N^3\int_0^t\mathbf P_{x}(\widetilde X_{s}^{(1)}= 1) ds
+6N^3\int_0^t\mathbf P_{x}(\widetilde X_{s}^{(1)}= N-1) ds
\le \max_{x,y\in\Lambda_{N-1}}
\big( g(x)-g(y)\big), 
\end{split}
\end{equation*}
which is bounded by $CN^2$ with some $C>0$, and thus the desired result follows. 
\end{proof}

\begin{lemma}
\label{lem:SRW_transition_prob_estimate}
Let $\{Y_t^{(1)}: t\ge0 \}$ be the simple symmetric random walk on $\mathbb T_N$ with jump rate $N^2$ to all directions, starting from $x\in\mathbb T_N$ and let $P_x$ be the associated probability measure.
There exists a constant $C>0$ such that 
\begin{equation*}
\max_{x,y\in\mathbb T_N} |\Delta_N^{(1)} P_x(Y_s^{(1)}=y)|
\le \frac{C}{Nt^{3/2}}
\end{equation*}
and 
\begin{equation*}
\max_{x,y\in\mathbb T_N} |\nabla^N P_x(Y_s^{(1)}=y)|
\le \frac{C}{Nt}. 
\end{equation*}
\end{lemma}
\begin{proof}
Define $p_t^x(y)\coloneqq P_x(Y_t^{(1)}=y)$ to simplify the notation.  
Then, we have that $p_t^x(y)$ solves the following evolution equation:
\begin{equation}
\label{eq:evolution_RW}
\begin{cases}
\partial_tp_t^x(y)=N^2\big( p^x_t(y+1)+ p^x_t(y-1) -2p^x_t(y) \big) , \\
p_0^x(y)=\delta_{x,y}
\end{cases}
\end{equation}
where $\delta_{x,y}$ stands for the Kronecker delta. 
Let us firstly focus on the assertion for the Laplacian. 
By standard computations with the Fourier transform we have that for any $f\in \ell^2(\mathbb T_N)$, and for any $k\in\mathbb T_N$, 
\begin{equation*}
\widehat{\Delta_N^{(1)}f}(k)
=-\Lambda_N(k/N) \widehat{f}(k),
\end{equation*}
where
\begin{equation*}
\Lambda_N(k/N) = 4N^2\sin^2(\pi k/N). 
\end{equation*}
Hence, applying the Fourier transform in equation \eqref{eq:evolution_RW}, we obtain that
\begin{equation*}
\widehat{p_t^x}(k)
=\widehat{p^x_0}(k) e^{-\Lambda_N(k/N)t}
= e^{-2\pi i kx/N} e^{-\Lambda_N(k/N)t}.
\end{equation*}
By inverting the last display, we have that 
\begin{equation*}
p^x_t(y)
=\frac{1}{N} \sum_{k\in\mathbb T_N} e^{-\Lambda_N(k/N)t -2\pi ik(x-y)/N}. 
\end{equation*}
Hence, computing the one-dimensional Laplacian we obtain
\begin{equation*}
\begin{aligned}
\Delta^{(1)}_N p_t^x(y)
&=N\sum_{k\in\mathbb{T}_N} e^{-t\Lambda_N\big(\tfrac kN\big)-\tfrac{2\pi ik(x-y)}{N}} 
\Big( e^{\tfrac{2\pi ik}{N}}+e^{\tfrac{-2\pi ik}{N}}-2\Big)\\ 
&= -\frac{1}{N} \sum_{k\in\mathbb{T}_N} \Lambda_N\big( \tfrac kN)
e^{-t\Lambda_N\big(\tfrac kN\big)-\tfrac{2\pi ik(x-y)}{N}} .
\end{aligned}
\end{equation*}
Noting $x-x^3/6 \le \sin x\le x$ for any $x\in [0,\pi]$, we have that 
\begin{equation*}
\begin{aligned}
|\Delta^{(1)}_N p_t^x(y)|
&\le \frac{C}{N}\sum_{k\in\mathbb T _N} k^2e^{-k^2t}
\le \frac{C}{N} \bigg( e^{-t} + \int_1^t s^2e^{-s^2t}ds \bigg)
\end{aligned}
\end{equation*}
where $C>0$ is a universal constant independent of $N$. 
Here, since the integral in the utmost right-hand side of the last display is bounded by $Ct^{-3/2}$ with another universal constant $C>0$, we conclude that this integral gives the worst bound to conclude the proof for the first assertion.
The other assertion for the divergence is analogous. 
A crucial estimate for this case would be: 
\begin{equation*}
|\nabla^Np_t^x(y)|
\le \frac{C}{N}\sum_{k\in\mathbb T_N}|k|e^{-Ck^2t/N^2}
\le \frac{C}{N}\int_{-N/2}^{N/2}|u|e^{-Ctu^2}du
\le \frac{C}{Nt},
\end{equation*}
for some universal constant $C>0$, and thus we conclude the proof.  
\end{proof}

\subsection*{Acknowledgments}
P.G. expresses warm thanks to Fundação para a Ciência e Tecnologia FCT/Portugal for
financial support through the projects UIDB/04459/2020, UIDP/04459/2020 and ERC/FCT. J.P.M. thanks CNPq for the financial support through the PhD scholarship and FCT Portugal for the support through the project ERC/PT Super-diffusive.
K.H. thanks the hospitality of Instituto Superior T\'ecnico for his research visit during the period February-March 2025 when part of this work was developed. 

\bibliographystyle{abbrv}
\bibliography{ref} 

\end{document}